\documentclass[12pt]{amsart}
\usepackage{amssymb,amsmath,amsfonts,latexsym,mathrsfs}
\usepackage{bm}

\usepackage{array,graphics,color}
\numberwithin{equation}{section}

\newtheorem{propn}{Proposition}[section]
\newtheorem{thm}[propn]{Theorem}
\newtheorem{lemma}[propn]{Lemma}

\newtheorem{cor}[propn]{Corollary}
\newtheorem{claim}[propn]{Claim}
\newtheorem*{thm*}{Theorem}
\theoremstyle{definition}
\newtheorem{defn}[propn]{Definition}

\newtheorem*{examples}{Examples}

\newtheorem{rem}{Remark}[section]

\newtheorem*{question}{Question}

\newcommand{\Nat}{\mathbb{N}}

 \newcommand{\D}{\mathbb{D}}

 \DeclareMathOperator{\ran}{ran}
 \DeclareMathOperator{\mult}{Mult}

\newcommand{\cla}{\mathcal{A}}
\newcommand{\clb}{\mathcal{B}}

\newcommand{\cld}{\mathcal{D}}
\newcommand{\cle}{\mathcal{E}}
\newcommand{\clf}{\mathcal{F}}

\newcommand{\clh}{\mathcal{H}}
\newcommand{\clk}{\mathcal{K}}
\newcommand{\cll}{\mathcal{L}}
\newcommand{\clm}{\mathcal{M}}

\newcommand{\clo}{\mathcal{O}}
\newcommand{\clp}{\mathcal{P}}
\newcommand{\clq}{\mathcal{Q}}

\newcommand{\cls}{\mathcal{S}}

\newcommand{\clw}{\mathcal{W}}
\newcommand{\clx}{\mathcal{X}}

\newcommand{\crh}{\mathscr{H}}

\newcommand{\raro}{\rightarrow}

\begin{document}

\title[Pure contractive multipliers]{Pure contractive multipliers of some reproducing kernel Hilbert spaces and Applications}


\author[Sarkar]{Srijan Sarkar}
\address{Department of Mathematics, Indian Institute of Science, Bangalore, 560012, India}
\email{srijans@iisc.ac.in,
srijansarkar@gmail.com}



\begin{abstract}
A contraction $T$ on a Hilbert space $\mathcal{H}$ is said to be \textit{pure} if the sequence $\{T^{*n}\}_{n}$ converges to $0$ in the strong operator topology. In this article, we prove that for contractions $T$, which  commute with certain tractable tuples of commuting operators $X = (X_1,\ldots,X_n)$ on $\mathcal{H}$, the following statements are equivalent:
\begin{enumerate}
\item[(i)] $T$ is a pure contraction on $\mathcal{H}$,
\item[(ii)] the compression $P_{\mathcal{W}(X)}T|_{\mathcal{W}(X)}$ is a pure contraction,
\end{enumerate}
where $\mathcal{W}(X)$ is the wandering subspace corresponding to the tuple $X$.

An operator-valued multiplier $\Phi$ of a vector-valued reproducing kernel Hilbert space (rkHs) is said to be \textit{pure contractive} if the associated multiplication operator $M_{\Phi}$ is a pure contraction. Using the above result, we find that operator-valued mulitpliers $\Phi(\bm{z})$ of several vector-valued rkHs's on the polydisc $\mathbb{D}^n$ as well as the unit ball $\mathbb{B}_n$ in $\mathbb{C}^n$ are pure contractive if and only if $\Phi(0)$ is a pure contraction on the underlying Hilbert space. The list includes Hardy, Bergman and Drury-Arveson spaces. Finally, we present some applications of our characterization of pure contractive multipliers associated with the polydisc.
\end{abstract}
\maketitle

\section{Introduction}
Let $\cle$ be a Hilbert space and $\clh_K(\cle)$ be a $\cle$-valued normalized reproducing kernel Hilbert space (abbreviated as rkHs) on some domain $\Omega \subset \mathbb{C}^n$. Our study is based on certain functions called \textit{mutliplers} of $\clh_K(\cle)$. The collection of multipliers, denoted by $\mult (\clh_K(\cle))$, consists of operator-valued functions $\Phi$ on $\Omega$ for which $\Phi f \in \clh_K(\cle)$ for all $f \in \clh_K(\cle)$. Associated to a multiplier, there always exists a multiplication operator on the rkHs. This forms a strong interplay between function theory and operator theory and our objective stems from this connection. More precisely, we are interested in the membership of mulitplication operators (arising from multipliers) in the following class of operators.
\begin{defn}
A contraction $T$ on $\clh$ (that is, $\|T\| \leq 1$) is said to be pure if $T^{*m} \raro 0$ in the strong operator topology (that is, for all $h \in \clh, \|T^{*m}h\| \raro 0)$ as $m \raro \infty$.
\end{defn}
Shift operators on several rkHs's are natural examples of pure contractions. In particular, shift operator on vector-valued Hardy spaces play a fundamental role. For a Hilbert space $\cle$, the $\cle$-valued Hardy space is defined by
\[
H_{\cle}^2(\D) := \{\sum_{n=0}^{\infty} a_n z^n \in \clo(\D;\cle): a_n \in \cle, n \in \Nat, \sum_{n=0}^{\infty} \|a_n\|_{\cle}^2 < \infty\},
\] 
where, $\mathbb{D}$ is the open unit disc in $\mathbb{C}$. The scalar-valued Hardy space is denoted by $H^2(\D)$ and the collection of multipliers turns out to be the set of all bounded analytic functions on $\D$, denoted by $H^{\infty}(\D)$. $\mult (H_{\cle}^2(\D))$ is precisely the operator-valued bounded analytic functions on $\D$, denoted by $H_{\clb(\cle)}^{\infty}(\D)$, where $\clb(\cle)$ stands for bounded operators on $\cle$.

Using the reproducing property of the Szeg\"o kernel of $H_{\cle}^2(\D)$, one can easily show that $M_z \in \clb(H_{\cle}^2(\D))$ is a pure contraction. In fact, the remarkable Sz-Nagy-Foias model (see \cite{NF}) for a contraction proves that \textsf{a pure contraction upto unitary equivalence is the compression operator $P_{\clq} M_z|_{\clq}$ where, $\clq$ is a $M_z^*$ invariant closed subspace of some vector-valued Hardy space}. An important characterization that follows from the above model shows that: \textsf{a contraction $T$ on $\clh$ is pure if and only if its characteristic function $\Theta_T(z)$ is a \textit{inner} function} (see \cite{NF}). Although this is a complete characterization, determining the behaviour of $\Theta_T$ is a challenging problem in itself and, thus, it is seldom used for checking pure contractions. Hence, it is a natural question to find more examples of pure contractions and more importantly, to find a tractable criterion. Before stating our main question, let us establish the  following convention.
\begin{defn}
For a rkhs $\clh_K(\cle)$, a multiplier $\Phi \in \mult(\clh_K(\cle))$ is said to be \textit{contractive} if $M_{\Phi}$ is contraction on $\clh_K(\cle)$. Furthermore, $\Phi \in \mult (\clh_K(\cle))$ is said to be \textit{pure contractive} if $M_{\Phi}$ is a pure contraction on $\clh_K(\cle).$
\end{defn}

With this terminology in mind, let us state the problem of our interest.
\begin{question}
Given a $\cle$-valued reproducing kernel Hilbert space $\clh_K(\cle)$, find a tractable necessary and sufficient condition for pure contractive multipliers in $\mult (\clh_K(\cle))$.
\end{question}

An answer to the above question for vector-valued Hardy spaces on the disc can be derived from a result by Bercovici et al. (Lemma 3.3, \cite{BDF}). More precisely, \textsf{a contractive operator-valued analytic function $\Phi(z) \in H_{\clb(\cle)}^{\infty}(\D)$ is pure contractive if and only if $\Phi(0)$ is a pure contraction on $\cle$}. Recent works show that similar conclusions hold true for Toeplitz operators on $H^2(\D^n)$ (see \cite{CIL}). 
It turns out that, to extend this result to a wider class of vector-valued rkHs's on the polydisc, it is important to identify the role played by shift operators. In the case of some well-studied rkHs's we know that the shift operators exhibit properties like left-invertibility or being a pure isometry. The following result, obtained recently by the author, shows that it is indeed the fact that the shift operator, on a vector-valued Hardy space on the disc, is a pure isometry that leads to the result by Bercovici et al. (see Corollary 3.3, \cite{SS}).
\begin{propn}[paraphrasing [Proposition 3.2, \cite{SS}]\label{pure_mult1}
Let $T$ be a contraction on a Hilbert space $\clh$ and let $V$ be any pure isometry on $\clh$ such that $TV=VT$. Then the following are equivalent:
\begin{enumerate}
\item[(i)] $T$ is a pure contraction on $\clh$,
\item[(ii)] $P_{\clw(V)} T|_{\clw(V)}$ is a pure contraction on $\clw(V)$,
\end{enumerate}
where, $\clw(V):= \clh \ominus VH$ is the wandering subspace of $V$.
\end{propn}

In general, wandering subspaces of operators on Hilbert spaces play an important role in connection with invariant subspaces (see \cite{Halmos, Shimorin}). The above result is another instance where, the importance of wandering subspace becomes apparent. In other words, it forms a tractable subset of the Hilbert space $\clh$ that can be used to determine certain pure contractions. Let us give a few more definitions that are required for describing our main result.

\begin{defn}
For a tuple of operators $X = (X_1,\ldots,X_n)$ on a Hilbert space $\clh$, the wandering subspace is defined by
\[
\clw(X) := \bigcap_{i=1}^n \ker X_i^* = \clh \ominus \sum_{i=1}^n X_i \clh.
\]
\end{defn}

\begin{defn}
A tuple of commuting operators $X=(X_1,\ldots,X_n)$ on $\clh$ is said to possess the wandering subspace property if $\clh = \bigvee_{\bm{k} \in \Nat^n} X^{\bm{k}} \clw(X)$.
\end{defn}

Returning to our discussion on multipliers: the reader might expect that pure contractive multipliers in higher dimensions (for instance, Hardy space on polydisc) can also be derived from Proposition \ref{pure_mult1}. However, there exists the following  difficulty: for a tuple of operators $X= (X_1,\ldots,X_n)$ on $\clh$ the wandering subspace $\clw(X_i)$ for any $i \in \{1,\ldots,n\}$ is a significantly larger set containing $\clw(X)$. For instance, if we consider the simple example of $M_{\bm{z}} = (M_{z_1},\ldots,M_{z_n})$ on $H^2(\D^n)$ (the Hardy space on polydisc), then
\[
\clw(M_{\bm{z}}) = \mathbb{C}, \quad  \text{whereas} \quad \clw(M_{z_i}) \cong H^2(\D^{n-1}) \quad (i \in \{1,\ldots,n\}).
\]
Another difficulty that lies in higher dimensions is the absence of a Beurling-Lax-Halmos type result for shift-invariant closed subspaces, which plays an important role in proving Proposition \ref{pure_mult1}. Thus, it is evident that we require new techniques for proving an analogous result that can be applied to higher dimensions. Here, we follow an approach popularized by the methods of Halmos (in \cite{Halmos}) and Shimorin (in \cite{Shimorin}) by focussing on abstract tuples of operators on Hilbert spaces that behave like shift operators on several rkHs's. More precisely, we are able to extend Proposition \ref{pure_mult1} to pure contractions commuting with certain tuples of left-invertible operators.

\begin{thm}\label{main1}
Let $T$ be a contraction on $\clh$ and let $X=(X_1,\ldots,X_n)$ be a $n$-tuple of doubly commuting left-invertible operators on $\clh$ such that $X_i T = T X_i $, for all $i \in \{1,\ldots,n\}$. Moreover, assume that both $X$ and the corresponding $n$-tuple of Cauchy-duals $X^{\prime}= (X_1^{\prime},\ldots,X_n^{\prime})$ possess the wandering subspace property on $\clh$. Then the following are equivalent:
\item[(i)] $T$ is a pure contraction on $\clh$.
\item[(ii)]$P_{\clw(X)} T|_{\clw(X)}$ is a pure contraction on $\clw(X)$.
\end{thm}

An interesting thing to note in the above result is that the role of the wandering subspace property becomes more evident. In Proposition \ref{pure_mult1}, this property was not assumed because a pure isometry always satisfies the wandering subspace property due to the Wold-von Neumann decomposition.

Now, our main idea is to leverage Theorem \ref{main1} in the setting of rkHs's. It is well-known that the property of left-invertibility is satisfied by shift operators on several rkHs's, including the Hardy, (un-weighted and certain weighted) Bergman and Dirichlet spaces on the unit disc (see \cite{Halmos}, \cite{Shimorin}, \cite{Richter}). This property extends to tuples of doubly commuting shift operators on these rkHs's on the unit polydisc as well. This is achieved through the canonical identification of these spaces with the the tensor product of rkHs's on the unit disc. Using this identification and Theorem \ref{main1}, we completely characterize pure contractive multipliers of Hardy/Bergman/Dirichlet spaces on the unit polydisc.

\begin{thm}\label{main3}
For a Hilbert space $\cle$, let $\clh_k \otimes \cle$ be the $\cle$-valued Hardy/Bergman/ Dirichlet space on the polydisc. Then for any contractive $\Phi(\bm{z}) \in \mbox{Mult}(\clh_k \otimes \cle)$, the following statements are equivalent:
\begin{enumerate}
\item[(i)] $M_{\Phi}$ is  a pure contraction on $\clh_k \otimes \cle$.
\item[(ii)] $\Phi(0)$ is a pure contraction on $\cle$.
\end{enumerate}
\end{thm}

There are several rkhs's on the open unit ball $\mathbb{B}_n \subset \mathbb{C}^n$, $n \geq 2$, to which we are able to extend the above result by using a different approach. The list includes Drury-Arveson, Hardy, Bergman spaces as well as unitarily invariant \textit{complete Nevanlinna-Pick} spaces (see Theorems \ref{Berg_char} and \ref{cnp_char}). Shift operators on these rkHs's do not have properties like left-invertiblity and therefore, the method used in Section \ref{Sec3} completely fails here. Thus, we look back to Proposition \ref{pure_mult1} and consider an approach that is similar to the proof of Proposition \ref{pure_mult1} by looking for characterizations of shift invariant subspaces. A remarkable result by Clou\^atre et al. shows that a Beurling-Lax-Halmos type characterization indeed holds true for spaces with \textit{cnp} factors  (\cite{CHS}). This result plays an important role in proving our main result for the above-mentioned rkHs's.

Let us now discuss how this article is arranged. Excluding this present section, the rest of the paper is organized as follows. Section \ref{Sec2} contains notations, definitions and a brief overview of several results that we will use in this article. In Section \ref{Sec3}, we prove Theorem \ref{main1}, followed by proving Theorem \ref{main3}. In Section \ref{Sec4}, we shift our focus to rkHs's on $\mathbb{B}_n$. We prove a result that is analogous to Theorem \ref{main3}, for pure contractive mulitpliers of vector-valued Drury-Arveson/Hardy/Bergman spaces and  unitarily invariant \textit{complete Nevanlinna-Pick} spaces on $\mathbb{B}_n$. We conclude the latter section by applying our results to obtain a characterization for pure contractions belonging to the commutant of certain tuples of commuting operators for which a Wold-von Neumann decomposition type theorem has been established by Eschmeier and Langend\"orfer in \cite{EL}. Finally, in Section \ref{Sec5}, we apply the results in Section \ref{Sec3} to obtain pure isometric dilation of certain important  tuples of commuting contractions appearing in recent papers by Barik et al. (see \cite{BDHS} and \cite{BDS}).

\section{Notations and Preliminaries}\label{Sec2}
The aim of this section is to make this article as much self-contained as possible by recalling definitions and results that are important for this article. We also prove some new results at the end of this section, that will be mainly used in Section \ref{Sec4}.

For any $n$-tuple of bounded operators $X=(X_1,\ldots,X_n)$ on $\clh$, a subspace $\cls \subset \clh$ is said to be $X$-joint invariant if $T_i  \cls \subseteq \cls$ for $i=1,\ldots,n$. Let $\bm{m}=(m_1,\ldots,m_n) \in \Nat^n$, then 
\[
|\bm{m}|:= m_1 + \ldots + m_n; \quad \bm{m}!: = m_1 ! \cdots m_n ! ; \quad T^{\bm{m}}:= T_1^{m_1} \cdots T_n^{m_n}.
\]
\subsection{Left-invertible operators}
In this subsection is to present a brief overview of the results developed by Shimorin in his remarkable paper \cite{Shimorin}. We recommend the Ph.D. thesis of DeSantis for a comprehensive treatise on this topic \cite{Derek}.

The roots of this discussion lies in the classical work of Halmos \cite{Halmos} in which he established the Beurling's theorem for vector-valued Hardy spaces. He used the Wold-von Neumann decomposition theorem for isometries, which states that an isometric operator $V$ on $\clh$ always admits a decomposition $V= V|_{\clh_0}\oplus V|_{\clh_1}$ where, $\clh_0,\clh_1$ are reducing subspaces for $V$ such that $V|_{\clh_{0}}$ is a unitary, $V|_{\clh_1}$ is a pure isometry and $\clh = \clh_{0} \oplus \clh_{1}$. Furthermore, the subspaces $\clh_0, \clh_1$ have the following precise structure
\[
\clh_{0} = \bigcap_{m=1}^{\infty} V^m \clh ; \quad \clh_1 = \bigoplus_{m=0}^{\infty} V^{m} \clw(V). 
\]
In particular, if $\clh_0 = \{0\}$ (equivalently, $V$ is a pure isometry) then, $V$ is unitarily equivalent to $M_z$ on $H_{\clw(V)}^2(\D)$. Motivated by this result, we have the following definition.
\begin{defn}
A tuple of operators $X=(X_1,\ldots,X_n)$ on $\clh$ is said to be analytic if $\bigcap_{l=0}^{\infty} \sum_{|\bm{m}| = l} X^{\bm{m}} \clh = \{0\}$. 
\end{defn}
%
In the case of a single operator $T \in \clb(\clh)$ this condition turns out to be $\bigcap_{m=0}^{\infty} T^{m} \clh = \{0\}$.

The above decomposition result for isometries was extended in a certain sense by Shimorin in \cite{Shimorin}. He showed that operators $T$ on $\clh$, satisfying any one of the following conditions
\begin{enumerate}
\item[(i)] $T^{*2}T^2 - 2 T^* T + I_{\clh} \leq 0$;
\item[(ii)] $\|Tx + y \|^2 \leq 2(\|x\|^2 + \|Ty\|^2)$.
\end{enumerate}
admits a Wold-von Neumann type decomposition. In particular, if $T$ is assumed to be ananlytic, then it has the wandering subspace property on $\clh$. It is important to note that operators satisfying any one of the above conditions are always left-invertible. Let us recall some well-known characterizations for general left-invertible operators (see \cite{Shimorin} and it's list of references).
\begin{propn}\label{left1}
Let $T$ be a bounded operator on $\clh$. Then the following are equivalent
\begin{enumerate}
\item[(i)] $T$ is a left-invertible operator,
\item[(ii)] $T$ is bounded below,
\item[(iii)] $T^*T$ is invertible.
\end{enumerate}
\end{propn}
Shimorin in \cite{Shimorin} defined the the Cauchy-dual of a left-invertible operator $T$ to be 
\begin{equation}
T^{\prime} := T(T^*T)^{-1}.
\end{equation}
It is a simple but important observation that $T^{\prime}$ is also a left-invertible operator on $\clh$ and the Cauchy dual of $T^{\prime}$ is $T$.
\begin{propn}\label{left_1}
Let $T$ be a left invertible operator, then the following holds
\begin{enumerate}
\item[(i)] $T(T^*T)^{-1}T^*$ is a orthogonal projection onto $T\clh$.
\item[(ii)] $\ker T^{\prime *} = \ker T^*$.
\item[(iii)] $T^{\prime}\clh = T \clh$.
\end{enumerate}
\end{propn}

In [Proposition 2.7, \cite{Shimorin}], Shimorin remarkably observed that $T$ and its Cauchy dual $T^{\prime}$ has the following dual properties:
\begin{equation}\label{wand_eqn}
\bigvee_{m \in \Nat} T^m \ker T^* = (\underset{m \geq 1}\bigcap T^{\prime m} \clh)^{\perp}; \bigvee_{m \in \Nat} T^{\prime m} \ker T^* = (\underset{m \geq 1}\bigcap T^m \clh)^{\perp}.
\end{equation}

This duality between $T$ and $T^{\prime}$ establishes the following result on wandering subspace property (see Corollary 2.8, \cite{Shimorin}).
\begin{cor}\label{left_3}
A left-invertible operator $T$ on $\clh$ possesses the wandering subspace property if and only if its Cauchy dual $T^{\prime}$ is analytic.
\end{cor}

\subsection{Reproducing kernel Hilbert spaces and Multipliers}
Let $\Lambda$ be a set and $\cle$ be a Hilbert space. A operator-valued function $K: \Lambda \times \Lambda \raro \clb(\cle)$ is said to be a \textit{kernel} function if  for any $n \in \Nat$, $\{\lambda_1,\ldots \lambda_n\} \subseteq \Lambda$ and $\{\eta_1,\ldots,\eta_n\} \subseteq \cle$, the following positivity condition is satisfied
\[
\sum_{i,j=1}^n \langle K(\lambda_i,\lambda_j) \eta_j, \eta_i \rangle \geq 0.
\] 
The above positive semi-definite condition is often denoted by $K \geq 0$. Associated to this kernel function $K$, there exists a unique Hilbert space of $\cle$-valued analytic functions denoted by $\clh_K(\cle)$. In particular, the collection $\{ K(\cdot,\lambda)\eta: \lambda \in \Lambda, \eta \in \cle \}$ forms a total set in $\clh_K(\cle)$ and satisfies the following reproducing property
\[
\langle f, K(\cdot, \lambda) \eta \rangle_{\clh_K(\cle)} = \langle f(\lambda), \eta \rangle_{\cle} \quad (f \in \clh_K(\cle), \eta \in \cle, \lambda \in \Lambda).
\]
From the above identity one can easily compute the norm of the kernel function to be 
\[
\|K(\cdot, \lambda) \eta\|_{\clh_K(\cle)}^2 = \|K(\lambda,\lambda)\eta\|^2.
\]
\textsf{A scalar-valued kernel is denoted by $k$ and the corresponding rkHs by $\clh_k$. In this case, the reproducing property turns out to be $f(\lambda) = \langle f, k(\cdot, \lambda) \rangle$}. A scalar-valued kernel $k$ on $\Lambda$ is said to be \textit{irreducible} if $k(x,y) \neq 0$ for any $x,y \in \Lambda$ and $k_x$ and $k_y$ are linearly independent functions for $x \neq y$. Moreover, $k$ is said to be \textit{normalized} at some point $\lambda_{0} \in \Lambda$ if $k(\cdot, \lambda_{0}) \equiv 1$. 
\begin{defn} 
A $\clb(\cle)$-valued kernel $K$ on set $\Lambda$ is called \textit{quasi-scalar} if $K=kI_{\cle}$, where $k$ is a scalar-valued kernel on $\Lambda$.
\end{defn}
For domains $\Omega \subset \mathbb{C}^n$, a rkHs $\clh_K(\cle) \subset \clo(\Omega,\cle)$ with a quasi-scalar kernel $K=kI_{\cle}$ can always be identified with the rkHs $\clh_k \otimes \cle$ via the unitary map $U: \clh_K(\cle) \raro \clh_k \otimes \cle$ defined by
\[
U(\bm{z}^{\bm{m}} \eta) = \bm{z}^{\bm{m}} \otimes \eta \quad (\bm{z} \in \Omega, \bm{m} \in \Nat^n, \eta \in \cle).
\]
Thus, in the case of quasi-scalar kernels, the corresponding rkHs will simply be denoted by $\clh_k \otimes \cle$. Let us give a list of important rkHs with quasi-scalar kernels.
\begin{examples}
In the following, $\cle$ is any Hilbert space and $n \in \Nat \setminus \{0\}$.
\begin{enumerate}
\item[(i)]$H_{\cle}^2(\D^n):$  $\cle$-valued Hardy space on unit polydisc with the Szeg\"o kernel on $\D^n$,
\[
\mathbb{S}(\bm{z},\bm{w}) = \prod_{i=1}^n (1 - z_i \bar{w}_i)^{-1} I_{\cle}.
\]
\item[(ii)] $A_{\alpha, \cle}^2(\D^n)$: $\cle$-valued weighted Bergman spaces on $\D^n$ with the Bergman kernel 
\[
K(\bm{z},\bm{w}) = \prod_{i=1}^n (1-z_i \bar{w}_i)^{-2+ \alpha}I_{\cle} \quad (-1< \alpha <+\infty).
\]
For $\alpha=0$, the $\cle$-valued Bergman space on $\D^n$ is simply denoted by $A_{\cle}^2(\D^n)$.
\item[(iii)] $\cld_{\cle}(\D^n):$ $\cle$-valued Dirichlet spaces on $\D^n$ with the Dirichlet kernel
\[
K(\bm{z},\bm{w}) = \prod_{i=1}^n \big(-\frac{1}{z_i \bar{w}_i} \ln (1-z_i \bar{w}_i)^{-1} \big) I_{\cle}
\]
\item[(iv)] $H_m(\mathbb{B}_n, \cle):$ for any integer $m \geq 0$, the $\cle$-valued rkHs on the unit ball $\mathbb{B}_n \subset \mathbb{C}^n$ with the kernel 
\[
K_m(\bm{z},\bm{w}) = (1 - \langle \bm{z},\bm{w}\rangle)^{-m} I_{\cle} \quad (\bm{z}, \bm{w} \in \mathbb{B}_n).
\]
In the case of $m=1$, $H_1(\mathbb{B},\cle)$ is the $\cle$-valued Drury-Arveson space denoted by $H_{n, \cle}^2$. For $m=n$, $H_n(\mathbb{B}_n, \cle)$ is the $\cle$-valued Hardy space on $\mathbb{B}^n$ denoted by $H_{\cle}^2(\mathbb{B}_n)$. And for $m=n+1$, $H_{n+1}(\mathbb{B}, \cle)$ is the (unweighted) Bergman space on $\mathbb{B}_n$ denoted by $A_{\cle}^2(\mathbb{B}_n)$.
\end{enumerate}
\end{examples}

Here we would like to return to our discussion on left-invertible operators and wandering subspace property for rkHs's $\clh_k$. It is well known, that in the following cases, $M_z$ is a left-invertible and analytic operator satisfying the wandering subspace property on $\clh_k$.
\begin{enumerate}
\item[(i)] $M_z \in \clb(H^2(\D))$ is an isometry and hence, left-invertible (see \cite{Halmos}). \vspace{1mm}
\item[(ii)] For $-1< \alpha \leq 0$, the shift operator $M_{z} \in \clb(A_{\alpha}^2(\D))$ satisfies the operator identity $\|M_z f + g\|^2 \leq 2 (\|f\|^2 + \|M_z g\|^2)$ for $f,g \in A_{\alpha}^2(\mathbb{D})$ (see [Theorem 6.15, \cite{HKZ}). \vspace{1mm}
\item[(iii)]  $M_{z} \in \clb(\cld(\D))$ satisfies $M_z^2 {M_z}^{*2} - 2 M_z M_z^* +I = 0$, in other words, by Agler's convention $M_z$ is a $2$-isometry and it follows that $M_z$ is left-invertible (see \cite{Richter}).
\end{enumerate}

We can state the following result based on the above discussion and Corollary \ref{left_3}.

\begin{propn}
Let $\clh_k$ be the Hardy/Bergman/Dirichlet space on $\D$, then both $M_z \in \clb(\clh_k)$ and its Cauchy dual $M_z^{\prime}$ are left-invertible operators and both of them possesses the wandering subspace property on $\clh_k$.
\end{propn}

\subsection{Multipliers}
Let $\cle_1, \cle_2$ be Hilbert spaces and $K_i: \Lambda \times \Lambda \raro \clb(\cle_i)$ be a kernel for $i=1,2$. A function $\Phi: \Lambda \raro \clb(\cle_1,\cle_2)$ is said to be a \textit{multiplier} if $\Phi f \in \clh_{K_2}(\cle_2)$ for all $f \in \clh_{K_1}(\cle_1)$. The collection of multipliers from $\clh_{K_1}(\cle_1)$ to $\clh_{K_2}(\cle_2)$ is denoted by $\mult(\clh_{K_1}(\cle_1), \clh_{K_2}(\cle_2))$. It follows from the closed graph theorem that $\Phi \in \mult(\clh_{K_1}(\cle_1), \clh_{K_2}(\cle_2))$ induces a bounded operator $M_{\Phi}$ from $\clh_{K_1}(\cle_1)$ to $\clh_{K_2}(\cle_2)$, which for all $\lambda \in \Lambda$ and $\eta \in \cle_1$ satisfies,
\[
M_{\Phi}^*(K_1(\cdot, \lambda)) = K_2(\cdot,\lambda) \Phi(\lambda)^* \eta,
\]
and the norm on $\mult(\clh_{K_1}(\cle_1), \clh_{K_2}(\cle_2))$ is defined by $\|\Phi\|:= \|M_{\Phi}\|_{op}$ (see \cite{AM} for more details). The closed unit ball of the collection of multipliers is denoted by $\mult_1(\clh_{K_1}(\cle_1), \clh_{K_2}(\cle_2))$. In the case of $K_1=K_2$ and $\cle_1 = \cle_2$, the collection of multipliers becomes a Banach algebra and is simply denoted by $\mult(\clh_{K_1}(\cle_1))$. 

There is an important class of kernels that are motivated from the classical Nevanlinna-Pick problem on the unit disc $\D$ about the existence of certain multipliers (see \cite{AM}, \cite{CH}).  

\begin{defn}
A rkhs $\clh_k$ on a set $\Lambda$ is said to be a \textit{complete Nevanlinna-Pick} space (or a \textit{cnp} space) if for all positive integers $l,m$, every collection of points $\{z_1,\ldots,z_m\} \in \Lambda$ and every choice of $l \times l$ complex matrices $W_1,\ldots,W_m$, the non-negativity of the block matrix
\[
[(k(\lambda_i,\lambda_j)(I - W_iW_j^*))]_{i,j=1}^n
\]
is sufficient for the existence of a multiplier $\Phi$ with $\|\Phi\| \leq 1$ and $\Phi(z_i) = W_i$, for $i=1,\ldots,n$.
\end{defn}
A series of fundamental results on cnp spaces is due to Agler and M\textsuperscript{C}Carthy. Building from their works, M\textsuperscript{C}Cullough$-$Quiggin discovered a tractable criterion for cnp spaces from which the following result follows (see Theorem 7.6, \cite{AM}).
\begin{thm}[M\textsuperscript{C}Cullough-Quiggin]
A normalized kernel $k$ is cnp if and only if $1-1/k \geq 0$.
\end{thm}

From the above result, it easily follows that  $H^2(\D)$, $\cld(\D)$ and $H_n^2$ are cnp spaces. Apart from these well studied rkHs's, we are also interested in a certain type of cnp spaces on $\mathbb{B}_n$. We recommend the article \cite{CH} and its references for the following details.

\begin{defn}
A \textit{unitarily invariant} space on $\mathbb{B}_n$ is a reproducing kernel Hilbert space with a kernel of the following form:
\[
k(\bm{z},\bm{w}) = \sum_{n=0}^{\infty} a_n \langle \bm{z},\bm{w} \rangle^n \quad (\bm{z}, \bm{w} \in \mathbb{B}_n),
\]
where $\{a_n\}_n$ is a sequence of strictly positive coefficients with $a_0=1$.
\end{defn}
By definition, a unitarily invariant kernel on $\mathbb{B}_n$ is always normalized at the point $0$. And for $\alpha \in \Nat^n \setminus \{0\}$, the collection $\{(a_{|\alpha|} \frac{|\alpha|!}{\alpha !})^{\frac{1}{2}} \bm{z}^{\alpha}\}$ forms an orthonormal basis for $\clh_k$. The following is a generalization of [Theorem 7.33, \cite{AM}], and we recommend [Lemma 2.3, \cite{CH}] for a direct proof.
\begin{lemma}
Let $\clh_k$ be a unitarily invariant rkhs on $\mathbb{B}_n \subset \mathbb{C}^n$ with kernel 
\[
k(\bm{z},\bm{w}) = \sum_{n=0}^{\infty} a_n \langle \bm{z}, \bm{w}\rangle^n.
\] 
Then $\clh_k$ is an irreducible cnp space if and only if 
\[
1 - \frac{1}{k(\bm{z},\bm{w})} = \sum_{n=0}^{\infty} b_n \langle \bm{z}, \bm{w} \rangle^n \quad (\bm{z}, \bm{w} \in \mathbb{B}_n)
\]
for some sequence $\{b_n\}_n$ of non-negative numbers.
\end{lemma}

An intrinsic problem that is associated to any rkHs is a description for $M_z$-invariant closed subspaces. The first of such a characterization is due to Beurling \cite{Beurling} in the case of $M_z$-invariant clsoed subspaces of $H^2(\D)$. This remarkable result holds true even for $A^2(\D)$ as proved by Aleman, Richter and Sundberg in \cite{ARS} and also for $\cld(\D)$ as proved by Richter (see \cite{Richter}). In particular, Beurling's theorem extracts a significant information about the subspace. 
\begin{thm}[Beurling]
$\cls$ is a $M_{z}$-invariant closed subspace of $H^2(\D)$ if and only if $\cls = \theta H^2(\D)$ for some inner function $\theta \in H^{\infty}(\D)$ (that is, $|\theta(\lambda)| = 1$ a.e. on $\mathbb{T}$).
\end{thm}

Results like above rarely holds in higher dimensions, for instance, Rudin showed that even in $H^2(\D^2)$ such a result does not exist (see \cite{Rudin}). However, in the case of unit ball there are several rkHs's which admit Beurling-type characterization of invariant subspaces. A recent result by Clou\^atre et al. in \cite{CHS} proved that rkHs's with cnp factors admit this property. The list of references in \cite{CHS} gives a brief account of several earlier results in this direction. For our interest, we combine [Theorem 1.1, Proposition 2.5, \cite{CHS}] to obtain the following result. 

\begin{thm}[\textit{Clou\^atre, Hartz and Schillo}]\label{Beurling_cnp}
Let $k$ be a kernel on $\mathbb{B}_n$ and let $s$ be a complete Nevanlinna-Pick kernel on $\mathbb{B}_n$ such that $k/s \geq 0$. Let $\cle$ be a Hilbert space and let $\clm$ be a non-zero closed subspace of $\clh_k \otimes \cle$. Then the following are equivalent:
\begin{enumerate}
\item[(i)] The subspace $\clm$ is $M_{\bm{z}}$-invariant,
\item[(ii)] There exists an auxillary space $\clf$ and a partially isometric multiplier \\ $\Gamma(\bm{z}) \in \mbox{Mult}(\clh_s \otimes \clf, \clh_k \otimes \cle)$ such that $
\clm = \Gamma(\bm{z}) (\clh_s \otimes \clf)$.
\end{enumerate}
\end{thm}

\subsection{Wandering subspaces} 
In this subsection, we are interested in establishing few results on wandering subspaces associated to certain tuples of operators. Let us begin by observing that wandering subspace associated to analytic tuples must be non-zero.

\begin{propn}\label{wander_opt}
Let $X=(X_1,\ldots,X_n)$ be a tuple of commuting operators on $\clh$ such that $\bigcap_{l=0}^{\infty} \sum_{|\bm{m}|=l} X^{\bm{m}} \clh = \{0\}$ and let $\clm$ be a $X$-joint invariant closed subspace of $\clh$. Then $\clm \neq \{0\}$ if and only if $\clw(X|_{\clm}) \neq \{0\}$. 
\end{propn}
\begin{proof}
It is obvious that $\clm= \{0\}$ implies that $\clw(X|_{\clm}) = \{0\}$. For proving the converse part, let us first observe that for any tuple of operators $X=(X_1,\ldots,X_n)$ on $\clh$, a subset $\cls \subset \clh$ and any $l \in \Nat$, we have the following
\[
\bigvee_{\underset{|\bm{m}| \geq l}{\bm{m} \in \Nat^n}} X^{\bm{m}} (\cls) = \bigvee_{\underset{|\bm{m}| \geq l}{\bm{m} \in \Nat^n}} X^{\bm{m}} (\overline{\cls}).
\]
Now, $\clw(X|_{\clm}) = \{0\}$, implies that $\clm = \overline{X_1 \clm + \ldots + X_n \clm}$ and therefore, for any $m \in \Nat$ we have
\[
\clm = \bigvee_{\bm{m}\in \Nat^n} X^{\bm{m}} (\clm) = \bigvee_{\underset{|\bm{m}| \geq 1} {\bm{m}\in \Nat^n}} X^{\bm{m}} (\clm) = \ldots =  \bigvee_{\underset{|\bm{m}| \geq l} {\bm{m}\in \Nat^n}} X^{\bm{m}} (\clm).
\]
Thus, for any $h \in \clm$ we always get that $h \in \bigcap_{l=0}^{\infty} \sum_{|\bm{m}|=l} X^{\bm{m}} \clh$ and the latter subspace is $\{0\}$ by our assumption. In other words, $\clm = \{0\}$. This completes the proof.
\end{proof}

Let us end this section by proving a generalization of  [Theorem 5.2, \cite{BEKS}], that gives an explicit description of wandering subspaces in certain cases.

\begin{lemma}\label{lemma_wander}
Let $A=(A_1,\ldots,A_n)$ and $B=(B_1,\ldots,B_n)$ be $n$-tuple of commuting operators on Hilbert spaces $\clh_1$ and $\clh_2$, respectively. Suppose, there exists a partial isometry $\Pi: \clh_1 \raro \clh_2$ such that $\Pi A_i = B_i \Pi$ for $i \in \{1,\ldots,n\}$. Then $\cls = \Pi \clh_1$ is a $B$-invariant closed subspace of $\clh_2$ and moreover,
\[
\clw(B|_{\cls}) := \cls \ominus \sum_{i=1}^n B_i \cls = \Pi ((\ker \Pi)^{\perp} \cap \clw(A)).
\]
\end{lemma}
\begin{proof}
Note that by definition, $\clw(B|_{\cls}) = \bigcap_{i=1}^n \ker P_{\cls}B_i^*|_{\cls}$. Since $\Pi$ is a partial isometry, $f \in \cls$ implies that $f= \Pi g$, for some $g \in \ran \Pi^*$. Furthermore, 
\begin{align*}
f \in \bigcap_{i=1}^n \ker P_{\cls}B_i^*|_{\cls} \Leftrightarrow  B_i^* f \perp \cls \quad (\text{ for all }i \in \{1,\ldots,n\}),
\end{align*}
In other words, $\Pi^* B_i^* f =0$ for all $i \in \{1,\ldots,n\}$. Again, 
\[
A_i^*g = A_i^* \Pi^*  \Pi g =  \Pi^* B_i^* f \quad (i \in \{1,\ldots,n\}),
\]
implies that $\Pi^* B_i^* f =0$  if and only if $A_i^*g = 0$ for all $i \in \{1,\ldots,n\}$. Thus, $g \in \clw(A)$ and therefore, $\clw(B|_{\cls}) = \Pi ((\ker \Pi)^{\perp} \cap \clw(A))$. This completes the proof.
\end{proof}

\section{Pure contractions commuting with left-invertible operators}\label{Sec3}
This section consists of three sub-sections. In the first sub-section, we prove Theorem \ref{main1}. In the later sub-sections, we apply this result to obtain our main characterization on multipliers of several reproducing kernel Hilbert spaces on $\mathbb{D}^n$.

Let us begin with a result on commuting orthogonal projections (see [Lemma 2.5,\cite{jaydeb1}])
\begin{lemma}\label{sum_projn}
If $(P_1,\ldots,P_n)$ is a $n$-tuple of commuting orthogonal projections on a Hilbert space $\clh$, then $\cll = \sum_{i=1}^n \ran P_i$ is a closed subspace and moreover,
\[
P_{\cll} = I_{\clh} - \prod_{i=1}^n (I - P_i) =  \bigoplus_{i=1}^n P_i \prod_{j > i }^n (I_{\clh} - P_j).
\]
\end{lemma}

\begin{defn}
A tuple of commuting operators $X=(X_1,\ldots,X_n)$ on $\clh$ is said to be \textit{doubly commuting} if $X_i^*X_j= X_jX_i^*$ for $i,j \in \{1,\ldots,n\}$ such that $i \neq j$.
\end{defn}

%
\textsf{For the purpose of our study in this section, we will assume that $X=(X_1,\ldots,X_n)$ is a tuple of doubly commuting left-invertible operators on $\clh$.}

It follows from a straightforward computation that if $X$ is a $n$-tuple of doubly commuting left-invertible operators, then the tuple of respective Cauchy duals $X^{\prime}=(X_1^{\prime},\ldots,X_n^{\prime})$, also forms a tuple of doubly commuting left-invertible operators.

Now, corresponding to the tuple $X$, we have a tuple of commuting orthogonal projections $\big(X_1 (X_1^*X_1)^{-1} X_1^*,\ldots ,X_n (X_n^*X_n)^{-1} X_n^* \big)$. From Proposition \ref{left1} and Proposition \ref{left_1}, we have $\ker X_i^* = \ker X_i^{\prime *}$ or equivalently, $X_i \clh = X_i^{\prime} \clh$ for all $i \in \{1,\ldots,n\}$ and therefore, by using Lemma \ref{sum_projn} we get
\begin{equation}\label{eqn1}
P_{\clw(X)^{\perp}} = P_{\sum_{i=1}^n X_i \clh}  = \bigoplus_{k=1}^n X_k (X_k^*X_k)^{-1} X_k^* \prod_{i > k }^n P_{\ker X_i^{*}}.
\end{equation}

Now, let us establish few results that are important for the sequel.

\begin{lemma}\label{vector2}
Let $X=(X_1,\ldots,X_n)$ be a $n$-tuple of doubly commuting left-invertible operators on $\clh$ such that
both $X$ and $X^{\prime}$ possess the wandering subspace property on $\clh$. If $\clm$ is a proper (that is, $\clm \neq \{0\}$ or $\clh$) $X$-joint invariant subspace of $\clh$ such that $\clw(X) \subseteq \clm^{\perp}$. Then there exists a non-zero vector $\eta \in \clm$ such that $X_i^* \eta \in \clm^{\bot}$, for $i=1,\ldots,n$. 
\end{lemma} 

\begin{proof}
The assumption $\clw(X) \subseteq \clm^{\perp}$ serves two purposes:
\begin{enumerate}
\item[(a)] $\clm$ is not $X$-reducing, since otherwise, $ \clm \perp \bigvee_{\bm{m} \in \Nat^n} X^{\bm{m}} \clw(X) = \clh$, which will contradict the assumption that $\clm$ is proper.
\item[(b)] for any $\eta \in \clm$, there exists some $i \in \{1,\ldots,n\}$ for which $X_i^* \eta \neq 0$ as otherwise, $\eta \in \clm^{\perp}$.
\end{enumerate}
Now, let us note that 
\[
\clw(X) = \bigcap_{i=1}^n \ker X_i^* = \bigcap_{i=1}^n \ker X_i^{\prime *} = \clw(X^{\prime}),
\]  
and therefore, it follows from our assumption that,
\[
\clh = \bigvee_{\bm{m} \in \Nat^n} X^{\prime \bm{m}}\clw(X),
\]
This implies that there must exist a non-zero vector $h \in \clw(X)$ and $\bm{m} \neq 0$ such that $X^{\prime \bm{m}} h \not\perp \clm$. Let us define
\[
\tilde{m}_1 := \min \{m_{1} \in \Nat: \clm \not\perp X^{\prime \bm{m}} h, \text{ where }\bm{m}=(m_1,\ldots,m_n) \in \Nat^n \}.
\]
Now, let us recursively define the other co-ordinates. For all $i \in \{2,\ldots,n\}$, let
\[
\tilde{m}_i= \min \{m_{i} \in \Nat: \clm \not\perp X^{\prime \bm{m}} h; \bm{m} = (\tilde{m}_1, \ldots,\tilde{m}_{i-1},m_i, \ldots,m_n) \text{ for }(m_i,\ldots,m_n) \in \Nat^{n-i} \},
\]
For $\tilde{\bm{m}} = (\tilde{m}_1,\ldots,\tilde{m}_n) \in \Nat^n$, there exists a non-zero vector $\eta \in \clm$, $\zeta \in \clm^{\perp}$ and  such that
\[
X^{\prime \tilde{\bm{m}}} h = \eta \oplus \zeta.
\]
Note that $\eta \in \clm$ implies that $X_i^* \eta  \neq 0$ for some $i \in \{1,\ldots,n\}$. For any $i \in \{ 1,\ldots,n \}$ for which $\tilde{k_i} = 0$ we have
\[
X_{i}^* \eta = X_i^* X^{\prime \tilde{\bm{m}}} h - X_{i}^* \zeta = X^{\prime \tilde{\bm{m}}} X_i^*h - X_{i}^* \zeta  = - X_{i}^* \zeta \in \clm^{\perp}.
\]
For any other $i \in \{ 1,\ldots,n \}$ we have 
\[
X_i^* \eta = X_i^* X^{\prime \tilde{\bm{m}}} h -  X_i^* \zeta = X^{\prime \tilde{\bm{m}}_i}h -   X_i^* \zeta \in \clm^{\perp},
\]
where, $\tilde{\bm{m}}_i = (\tilde{m}_1,\ldots, \tilde{m}_{i-1}, \tilde{m}_{i} - 1, \tilde{m}_{i+1},\ldots,\tilde{m}_{n})$. Since $\tilde{\bm{m}}_i  < \tilde{\bm{m}}$ it implies that $X^{\prime \tilde{\bm{m}}_i }h \in \clm^{\perp}$ and so $X_i^* \eta \in \clm^{\perp}$.  This completes the proof.
\end{proof}

\begin{rem}
It follows from the Beurling-Lax-Halmos theorem that any $M_{z}$-invariant closed subspace of $H_{\cle}^2(\D)$ is of the form $\clm = \Theta H_{\clf}^2(\D)$ for some Hilbert space $\clf$. This implies that for any $\eta \in \clf$ we get $M_z^* M_{\Theta} \eta \in \clm^{\perp}$. The above result shows that even though Beurling-Lax-Halmos type results do not hold for $H_{\cle}^2(\D^n)$, we can still obtain a vector $\eta$ in the $M_{\bm{z}}$-joint invariant closed subspace $\clm$ for which $M_{z_i}^* \eta \in \clm^{\perp}$ for all $i=1,\ldots,n$.
\end{rem}

\begin{lemma}\label{converge}
Let $A$ be a left-invertible operator on $\clh$ and let $\{h_n\}$ be a bounded sequence of vectors in $A \clh$. Then $\lim_{n \raro 0} \| A^* h_n \| = 0$ implies that $\lim_{n \raro 0} \| h_n \| = 0$.
\end{lemma}
\begin{proof}
By assumption, there exists a strictly positive real number $\alpha$ such that $\|h_n\| \leq \alpha$ for all $n \in \Nat$. Now, $h_n \in A\clh$ implies that there exists $y_n \in \clh$ such that $A y_n = h_n$, for all $n \in \Nat$. Since $A$ is bounded below there exists $c>0$ such that $\|y_n\| \leq c \|A y_n\| = c \|h_n\|$. Now,
\[
\lim_{n \raro 0} \|h_n\|^2 = \lim_{n \raro 0}\|A y_n\|^2 \leq \lim_{n \raro 0}\|A^* A y_n\| \|y_n\| \leq c \lim_{n \raro 0}\|A^* h_n\|  \|h_n\| \leq c \alpha \lim_{n \raro 0}\|A^* h_n\| = 0.
\]
This completes the proof.
\end{proof}

Given a contraction $T \in \clb(\clh)$, since $\{T^m T^{*m}\}_m$ is a monotonic decreasing sequence of  positive operators there exists a positive operator say $A_T$, such that $\{T^m T^{*m}\}_n$ converges in the strong operator topology to $A_T^2$. Moreover, by definition we have $T A_T^2 T^* = A_T^2$ which further implies that
\[
\|A_T T^{*m} h \| = \|A_T h \| \quad (h \in \clh, m \in \Nat).
\] 
We are now ready to prove the main result of this section.

\begin{proof}[Proof of Theorem \ref{main1}]
$(i) \implies (ii):$ Since $\clw(X)$ is a $T^*$-invariant subspace of $\clh$, for any $m \in \Nat$, we get 
\[
(P_{\clw(X)} T|_{\clw(X)})^{*m} =  T^{*m}|_{\clw(X)},
\]
which implies that $P_{\clw(X)} T|_{\clw(X)}$ is a pure contraction.

$(ii) \implies (i):$ Our aim is to prove that $\ker A_T = \clh$. Let us first observe that the assumption $P_{\clw(X)} T|_{\clw(X)}$ is a pure contraction on $\clw(X)$ implies that $\clw(X) \subseteq \ker A_T$. Indeed for any $\eta \in \clw(X)$
\[
\|A_T \eta\| = \lim_{m \raro \infty} \|T^{*m} \eta \| = \lim_{m \raro \infty} \| \big( P_{\clw(X)}T P_{\clw(X)})^{*m} \eta \| = 0.
\]

Case (I): If $\ker A_T$ is a $X$-joint reducing subspace of $\clh$, then by the above condition we get $\clh = \bigvee_{\bm{m} \in \Nat^n} X^{\bm{m}} \clw(X) \subseteq \ker A_T \subseteq \clh$. This will show that $\ker A_T = \clh$. 

Case (II): Here, we consider the case where $\ker A_T$ is not a $X$-joint reducing subspace of $\clh$. Let us assume that $\ker A_T$ is a proper subspace of $\clh$. Using Lemma \ref{vector2}, we find a non-zero element $\eta \in (\ker A_T)^{\perp}$ such that $X_i^* \eta \in \ker A_T$ for all $i \in \{1,\ldots,n\}$.  Now, the orthogonal decompositon in equation (\ref{eqn1}) shows that for each $m \in \Nat$, there exists $\zeta_m \in \clw(X)$ and $h_{j,m} \in X_j (X_j^*X_j)^{-1} X_j^* \prod_{i > j }^n P_{\ker X_i^*}$ such that
\begin{equation}\label{eqn2}
T^{*m} \eta =   \zeta_m \oplus  \sum_{j=1}^n h_{j,m},
\end{equation}
and therefore, 
\[
\|T^{*m} \eta \|^2 = \| \zeta_m\|^2 +  \sum_{j=1}^n \| h_{j,m} \|^2.
\]
Since $T$ is a contraction, we have $\|h_{j,m}\| \leq  \|\eta\|$ for all $j \in \{1,\ldots,n\}$ and $m \in \Nat$. If we act by $X_i^*$ on the left hand side of equation (\ref{eqn2}), then we get

\begin{equation}\label{eqn3}
X_i^* T^{*m} \eta = X_i^* \zeta_m + X_i^*  \sum_{j=1}^n h_{j,m} = X_i^* \sum_{j=1}^n h_{j,m}.
\end{equation}

If $i = n$, then by the orthogonal decomposition in equation (\ref{eqn1}), we have
\[
X_n^* T^{*m} \eta = X_{n}^* \sum_{j=1}^n h_{j,m} = X_{n}^* h_{n,m}.
\]
Thus, using the fact that $X_n^* \eta \in \ker A_T$, we get
\[
\lim_{k \raro \infty} \|X_{n}^* h_{n,m} \| = \lim_{k \raro \infty} \| X_n^* X^{*m}  \eta \| = \lim_{k \raro \infty} \|  T^{*m}  X_n^*  \eta  \| = 0.
\]
Now, $h_{n,m} \in X_n \clh$ and therefore, by applying Lemma \ref{converge}, we obtain $\lim_{m \raro \infty} \|h_{n,m}\|=0$.   We will now proceed to establish that $\lim_{m \raro \infty}\|\sum_{j=1}^n h_{j,m}\| = 0$.  For proving this, we will show that it is sufficient to obtain the following result.

\begin{claim}
If $ \underset{m \raro \infty}{\lim} \|\underset{{j=l+1}}{\overset{n}\sum} h_{j,m}\| = 0$, then $\underset{m \raro \infty}{\lim}
 \|\underset{{j=l}}{\overset{n}\sum} h_{j,m}\| = 0$ for any $l \in \{1,\ldots,n-1\}$.
\end{claim}

Let us observe from equation (\ref{eqn1}) and equation (\ref{eqn3}) that
\[
X_l^* T^{*m} \eta = X_l^* \sum_{j=1}^n h_{j,m} = X_l^* \sum_{j=l}^n h_{j,m},
\]
and thus,
\[
 \|X_l^* h_{l,m}\| =  \|X_l^* T^{*m} \eta - X_{l}^* \sum_{j=l+1}^n h_{j,m}\|,
\]
which further implies that $\lim_{m \raro \infty} \|X_l^* h_{l,m}\| = 0$.  Now, $h_{l,m} \in X_l \clh$ and again by applying Lemma \ref{converge}, we obtain $\lim_{m \raro \infty} \|h_{l,m}\|=0$.  This completes the proof of the claim.

Starting with $l=n$, if we iterate this argument for $n-1$ times, then we get $\|\sum_{j=1}^n h_{j,m}\| \raro 0$ as $k \raro \infty$. Now, for $k \in \Nat$ we have
\begin{equation}\label{pure_cont}
\| A_T \eta\| = \| A_T T^{*m} \eta \| =  \| A_T (\zeta_m \oplus \sum_{j=1}^n h_{j,m})\| \leq \|A_T\|  \|\sum_{j=1}^n h_{j,m}\| \leq \|\sum_{j=1}^n h_{j,m}\|.
\end{equation}
If we let $m \raro \infty$, then we have $A_T \eta = 0$, which further implies that $\eta \in \mbox{ ker } A_T$. However, by assumption $\eta \in (\mbox{ker }A_{T})^{\perp}$ and therefore, $\eta =0$. This contradicts the assumption that $\ker A_{T}$ is a proper subspace of $\clh$. Since $\ker A_T \supseteq \clw(X) \neq \{0\}$, the only possibility is that $\ker A_{T} = \clh$ that is, $T$ is a pure contraction on $\clh$. This completes the proof.
\end{proof}

\begin{rem}
The above result holds true even if we consider $T$ to be a power bounded operator (say with a bound $M$) for which $\{T^{*m}\}_k$ converges strongly as $m \raro \infty$. The only difference in the method is that the operator $A_T$ cannot be constructed since $\{T^mT^{*m}\}$ will no longer be a monotonic decreasing sequence. Instead, we will have to replace equation (\ref{pure_cont}) with the following argument:
\begin{align*}
\lim_{m \raro \infty} \| T^{*m} \eta\| = \lim_{m \raro \infty} \lim_{l \raro \infty} \| T^{*(l+m)} \eta\| &= \lim_{m \raro \infty} \lim_{l \raro \infty}  \| T^{*l} (T^{*m} \eta) \| \\
&= \lim_{m \raro \infty} \lim_{l \raro \infty}  \| T^{*l}(\zeta_m \oplus \sum_{j=1}^n h_{j,m})\| \\
& \leq \lim_{m \raro \infty} \lim_{m \raro \infty} \| T^{*l}\zeta_m\| + \lim_{l \raro \infty} \| T^{*l}(\sum_{j=1}^n h_{j,m})\| \\
&= 0 + \lim_{m \raro \infty} \lim_{l \raro \infty} \| T^{*l}(\sum_{j=1}^n h_{j,m})\| \\
&\leq 0 + M \lim_{m \raro \infty} \|\sum_{j=1}^n h_{j,m}\| = 0.
\end{align*}
\end{rem}
%

Now, let us establish a similar result for doubly commuting pure isometries.

\begin{cor}\label{Hardy_pure2}
Let $V=(V_1,\ldots,V_n)$ be a $n$-tuple of doubly commuting pure isometries on $\clh$. Let $T \in \clb(\clh)$ be such that $ TV_i = V_i T$, for all $i \in \{1,\ldots,n\}$. Then the following statements are equivalent:
\item[(i)] $T$ is a pure contraction on $\clh$.
\item[(ii)]$P_{\clw(V_i)} T|_{\clw(V_i)}$ is a pure contraction on $\clw(V_i)$ for some $i \in \{1,\ldots,n\}$.
\item[(iii)]$P_{\clw(V)} T|_{\clw(V)}$ is a pure contraction on $\clw(V)$.
\end{cor}
\begin{proof}
Statements $(i)$ and $(ii)$ are equivalent from [Proposition 3.1, \cite{SS}]. From [Theorem 2.3, \cite{jaydeb2}], we know that any $n$-tuple of doubly commuting pure isometries satisfies the wandering subspace property. The equivalence of $(i)$ and $(iii)$ now follows from this observation and Theorem \ref{main1}, since an isometry is also a left-invertible operator by definition.
\end{proof}

\subsection{Wandering subspace for tuples of operators on tensor product of Hilbert spaces}
For a $n$-tuple of left-invertible operators $X=(X_1,\ldots,X_n)$ on $\clh$, let us consider the tuple of left-invertible operators $\clx=(\clx_1,\ldots, \clx_n)$ on $\crh:=\underset{n \text{ times}}{\underbrace{\clh \otimes \ldots \otimes \clh}}$ defined by 
\[
\clx_i = \underset{i-1 \text{ times}}{ \underbrace{I_{\clh} \otimes \ldots \otimes I_{\clh}}} \otimes X_i \otimes  \underset{n-i \text{ times}}{ \underbrace{I_{\clh} \otimes \ldots \otimes I_{\clh}}}.
\]

It follows from a straightforward computation that the Cauchy dual of the left-invertible operator $\clx_i$ is 
\[
\clx_i^{\prime} = \underset{i-1 \text{ times}}{ \underbrace{I_{\clh} \otimes \ldots \otimes I_{\clh}}} \otimes X_i^{\prime} \otimes  \underset{n-i \text{ times}}{ \underbrace{I_{\clh} \otimes \ldots \otimes I_{\clh}}}.
\]
The corresponding tuple of Cauchy duals is denoted by $\clx^{\prime} = (\clx_1^{\prime} ,\ldots, \clx_n^{\prime})$. Now, let us compute the wandering subspace of these tuples of operators.

\begin{lemma}\label{left_wand}
Let $X=(X_1,\ldots,X_n)$ be a $n$-tuple of doubly commuting left-invertible operators on $\clh$. Then $\clw(\clx) = \clw(\clx^{\prime}) = \clw(X_1) \otimes \ldots \otimes \clw(X_n)$.
\end{lemma}
\begin{proof}
Using the property of left-invertible operators we obtain that
\[
P_{\ker \clx_i^*} = \underset{i-1 \text{ times}}{ \underbrace{I_{\clh} \otimes \ldots \otimes I_{\clh}}} \otimes P_{\ker X_i^*} \otimes  \underset{n-i \text{ times}}{ \underbrace{I_{\clh} \otimes \ldots \otimes I_{\clh}}}.
\]


Note that $\clw(\clx) = \bigcap_{i=1}^n \ker \clx_i^* = \clw(\clx^{\prime})$. 
Since $X$ is a tuple of doubly commuting operators, both $\clx$ and $\clx^{\prime}$, are doubly commuting, which further implies that $(P_{\ker \clx_1^*},\ldots,P_{\ker \clx_n^*})$ is a commuting tuple of orthogonal projections and therefore,
\[
P_{\clw(\clx^{\prime})} = P_{\bigcap_{i=1}^n \ker \clx_i^*} = \prod_{i=1}^n P_{\ker \clx_i^*} = P_{\ker X_1^*} \otimes \ldots \otimes P_{\ker X_n^*}.
\]
Thus, 
\[
\clw(\clx) = \bigcap_{i=1}^n \ker \clx_i^* = \ker X_1^* \otimes \ldots \otimes \ker X_n^* = \clw(X_1) \otimes \ldots \otimes \clw(X_n).
\]
This completes the proof.
\end{proof}

\begin{thm}\label{left_char}
Let $X=(X_1,\ldots,X_n)$ be a $n$-tuple of doubly commuting left-invertible operators on $\clh$ such that both $X_i$ and $X_i^{\prime}$ possess the wandering subspace property on $\clh$ for each $i \in \{1,\ldots,n\}$. Then both $\clx^{\prime}$ and $\clx$ possess the wandering subspace property on $\crh$.
\end{thm}
\begin{proof}
Our aim is to show that 
\[
\crh = \bigvee_{\bm{m} \in \Nat^n} \clx^{\prime \bm{m}} \clw(\clx^{\prime}).
\]
By Lemma \ref{left_wand}, $\clw(\clx^{\prime}) = \clw(\clx_1^{\prime}) \otimes \ldots \otimes \clw(\clx_n^{\prime})$, which implies that 
\[
\clh \otimes \clw(X_2^{\prime}) \otimes \ldots \otimes \clw(X_{n}^{\prime})  =  \bigvee_{m \in \Nat^n} \clx_1^{\prime m} \big(\clw(\clx_1^{\prime}) \otimes \ldots \otimes \clw(\clx_n^{\prime})) \big) \subseteq \bigvee_{\bm{m} \in \Nat^n} \clx^{\prime \bm{m}} \clw(\clx^{\prime}).
\]
Similarly,
\[
\clh \otimes \clh \otimes \clw(X_3^{\prime}) \otimes \ldots \otimes \clw(X_{n}^{\prime})  =  \bigvee_{m \in \Nat^n} \clx_2^{\prime m} \big(\clh \otimes \clw(\clx_2^{\prime}) \otimes \ldots \otimes \clw(\clx_n^{\prime}) \big) \subseteq \bigvee_{\bm{m} \in \Nat^n} \clx^{\prime \bm{m}} \clw(\clx^{\prime}).
\]

Following this process for another $(n-2)$ times gives
\[
\crh \subseteq \bigvee_{\bm{m} \in \Nat^n} \clx^{\prime \bm{m}} \clw(\clx^{\prime}) \subseteq \crh.
\]
This proves that $\clx^{\prime}$ possesses the wandering subspace property on $\crh$. Now $(\clx^{\prime})^{\prime} = \clx$ and $\clw(X_i^{\prime}) = \clw(X_i)$ for all $i \in \{1,\ldots,n\}$ implies that $\clx$ possesses the wandering subspace property on $\crh$. This completes the proof.
\end{proof}

\subsection{Reproducing Kernel Hilbert spaces on the unit polydisc}


Let $\clh_{\tilde{k}}$ be the Hardy/\\Bergman/Dirichlet space on unit disc $\mathbb{D}$, then $M_z \in \clb(\clh_{\tilde{k}})$ is a left-invertible operator. Moreover, [Proposition 2.7, \cite{Shimorin}] proves that both $M_z$ and the Cauchy dual $M_z^{\prime}$ possesses the wandering subspace property on $\clh_{\tilde{k}}$. For a Hilbert space $\cle$, let $\clh_{k} \otimes \cle$ denote the $\cle$-valued Hardy/Bergman/Dirichlet kernel on $\D^n$ that is, $k(\bm{z}, \bm{w}) = \Pi_{i=1}^n {\tilde{k}}(z_i, \bm{w_i})$ ($\bm{z}, \bm{w} \in \D^n$). It is well known that the map 
\[
U: \clh_{k}(\D^n) \times \cle \raro (\underset{n\text{-times}}{\underbrace{\clh_{\tilde{k}}(\D) \otimes \ldots \otimes \clh_{\tilde{k}}(\D)}}) \otimes \cle,
\]
defined by $ U(\bm{z}^{\bm{m}} \otimes  \eta) = (z_1^{m_1} \otimes \ldots \otimes z_n^{m_n}) \otimes \eta \quad (\bm{z} \in \D^n, \bm{m} \in \Nat^n, \eta \in \cle),
$ is a unitary and moreover, $U (M_{z_i} \otimes I_{\cle}) = (\tilde{M}_{z_i} \otimes I_{\cle}) U$, where $\tilde{M}_{z_i} = \underset{i\text{-times}}{\underbrace{I_{\clh_{\tilde{k}}} \otimes \ldots \otimes I_{\clh_{\tilde{k}}}}} \otimes M_z \otimes \underset{(n-i)\text{-times}}{\underbrace{I_{\clh_{\tilde{k}}} \otimes \ldots \otimes I_{\clh_{\tilde{k}}}}}$. Using this identification and Theorem \ref{left_char}, we can deduce the following result.
\begin{thm}\label{analytic_powers}
For a Hilbert space $\cle$, let $\clh_k(\D^n) \otimes \cle$ be the $\cle$-valued Hardy/Bergman/ Dirichlet space on the polydisc. Then both $\tilde{M}_{\bm{z}}^{\prime}:=(\tilde{M}_{z_1}^{\prime} \otimes I_{\cle}, \ldots,\tilde{M}_{z_n}^{\prime} \otimes I_{\cle})$ and $\tilde{M}_{\bm{z}}:=(\tilde{M}_{z_1} \otimes I_{\cle}, \ldots,\tilde{M}_{z_n} \otimes I_{\cle})$ possess the wandering subspace property on $\clh_k(\D^n) \otimes \cle$.
\end{thm} 

We are now ready to prove the main result of this section.

\begin{proof}[Proof of Theorem \ref{main3}]
Note that $\Phi(\bm{z}) \in \mbox{Mult} (\clh_k(\D^n) \otimes\cle)$ implies that $\Phi(\bm{z}) \in H_{\clb(\cle)}^{\infty}(\D^n)$ and therefore, $M_{\Phi} (M_{z_i} \otimes I_{\cle}) = (M_{z_i} \otimes I_{\cle}) M_{\Phi}$ for $i \in \{1,\ldots,n\}$. This implies that $(\tilde{M}_{z_i} \otimes I_{\cle}) U M_{\Phi} U^* = U M_{\Phi} U^* (\tilde{M}_{z_i} \otimes I_{\cle})$. Now, using Theorem \ref{main1} and Theorem \ref{analytic_powers} we obtain that $U M_{\Phi} U^*$ is a pure contraction on $(\clh_{\tilde{k}}(\D) \otimes \ldots \otimes \clh_{\tilde{k}}(\D)) \otimes \cle$ if and only if $P_{\clw(\tilde{M}_{\bm{z}})} (U M_{\Phi} U^*)|_{\clw(\tilde{M}_{\bm{z}})}$ is a pure contraction on $\clw(\tilde{M}_{\bm{z}})$.  By definition $U^* \clw(\tilde{M}_{\bm{z}}) = \clw(M_{\bm{z}}) = \cle$, which implies that $P_{\clw(\tilde{M}_{\bm{z}})} (U M_{\Phi} U^*)|_{\clw(\tilde{M}_{\bm{z}})}$ is a pure contraction on $\clw(\tilde{M}_{\bm{z}})$  if and only if $\Phi(0)$ is a pure contraction on $\cle$. This completes the proof.
\end{proof}

In the scalar-valued spaces, we have the following result.
\begin{thm}\label{main2}
Let $\clh_k$ be the scalar-valued Hardy/Bergman/Dirichlet space on $\D^n$. Then for any $\phi(\bm{z}) \in \mbox{Mult}_{1} (\clh_k)$, the following statements are equivalent:
\begin{enumerate}
\item[(i)] $M_{\Phi}$ is  a pure contraction on $\clh_K$.
\item[(ii)] $|\Phi(0)|<1$.
\end{enumerate}
\end{thm}
Thus, we obtain [Theorem 3.2, \cite{CIL}] as a corollary of our results. In the case of inner multipliers of the Hardy space on polydisc, our methods give an alternative proof of following result observed in [Theorem 4.1, \cite{DPS}].
\begin{cor}
Let $\theta(\bm{z}) \in H^{\infty}(\D^n)$ be a non-constant inner function (that is, $|\theta(\bm{z})| = 1$ a.e. on $\mathbb{T}^n$). Then $M_{\theta}$ is always a pure isometry on $H^2(\D^n)$.
\end{cor}
\begin{proof}
By the maximum principle, $|\theta(0)|<1$ for any non-constant inner function $\theta(\bm{z}) \in H^{\infty}(\D^n)$. Now, the proof follows by applying Theorem \ref{main2} for $H^2(\D^n)$.
\end{proof}

\begin{rem}
From our discussion on left-invertibility of shift operators in the Section \ref{Sec2}, it follows that  Theorem \ref{main3} can be extended for weighted Bergman spaces $A_{\alpha, \cle}(\mathbb{D}^n)$, where $-1 < \alpha \leq 0$, as well.
\end{rem}

We are now interested in finding a result involving the restriction of mulitplication operators on certain $M_{\bm{z}}$-joint invariant subspaces of the $H^2(\D^n)$. For proving this result, let us first recall that a $M_{\bm{z}}$-joint invariant closed subspace of $H^2(\D^n)$ is said to be \textit{doubly commuting} if  the commutator $[R_{z_i}^*,R_{z_j}] = 0$ for $i,j \in \{1,\ldots,n\}$ such that $i \neq j$, where $R_{z_i} := M_{z_i}|_{\cls}$. By [Theorem 1.3, \cite{jaydeb4}] any such $\cls$ is of the form $\theta H^2(\D^n)$ for some inner function $\theta \in H^{\infty}(\D^n)$. For the purpose of the next result, we shall call $\cls= \theta H^2(\mathbb{D}^n)$ to be a \textit{non-zero} based doubly commuting $M_{\bm{z}}$-joint invariant subspace of $H^2(\mathbb{D}^n)$, if $\theta(0) \neq 0$.

\begin{thm}\label{main4}
Let $\phi \in H^{\infty}(\D^n)$ be a contractive multiplier. Then $M_{\phi}$ is a pure contraction on $H^2(\D^n)$ if and only if there exists a non-zero based doubly commuting $M_{\bm{z}}$-joint invariant subspace of $H^2(\mathbb{D}^n)$ such that $M_{\phi}|_{\cls}$ is a pure contraction.
\end{thm}
\begin{proof}
Observe that $M_{\phi}$ is a pure contraction on $H^2(\D^n)$ will always imply $M_{\phi}|_{\cls}$ is a pure contraction for any $M_{\bm{z}}$-joint invariant closed subspace $\cls$. Now, if there exists such a subspace $\cls$ for which $M_{\phi}|_{\cls}$ is a pure contraction, then for any $f \in \cls$ we have
\[
\lim_{n \raro \infty} \|P_{\cls} M_{\Phi}^{*n} f\| = 0.
\]
In particular, for $f(z) = P_{\cls} k_{0}(\bm{z}) = P_{\cls}1$ (note that $\theta(0) \neq 0$ shows that $P_{\cls}1 \neq 0$) we have
\[
\|P_{\cls} M_{\phi}^{*n} P_{\cls}1\| = \|M_{\theta} M_{\theta}^* M_{\phi}^{*n} M_{\theta} M_{\theta}^* k_{0}(\bm{z}) \| = \|\overline{\theta (0)} M_{\phi}^{*n} k_{0}(\bm{z})\| = |\overline{\theta(0)} \hspace{1mm}\overline{\phi(0)}^{n}|.
\]
This implies that $|\overline{\phi(0)}^{n}| \raro 0$ as $n \raro \infty$ and therefore, by Corollary \ref{main2}, it follows that $M_{\Phi}$ is a pure contraction on $H^2(\D^n)$. This completes the proof. 
\end{proof}

\section{Pure contractive multipliers on the unit ball}\label{Sec4}
In this section, we obtain a characterization for pure contractive multipliers of several reproducing kernel Hilbert spaces on $\mathbb{B}_n$. For this purpose, we are interested in the following type of rkHs's.
\begin{defn}
A $\cle$-valued reproducing kernel Hilbert space $\clh_k \otimes \cle$ on $\mathbb{B}_n$ is said to be \textit{regular} if it satisfies the following conditions:\\
$(i) \hspace{1mm}\clh_k \otimes \cle \subset \clo(\mathbb{B}_n, \cle)$; \\
$(ii) \hspace{1mm} M_{\bm{z}}= (M_{z_1},\ldots,M_{z_n})$ is a tuple of bounded operators on $\clh_k \otimes \cle$;\\
$(iii)$  the collection of monomials form an orthogonal basis of $\clh_k \otimes \cle.$\\
$(iv)$ $\bigcap_{i=1}^n \ker M_{z_i}^* = \cle$.
\end{defn}

The author is grateful to Prof. Kunyu Guo at the Fudan University, China for a correspondence that led to establishing the following result.

\begin{propn}\label{wand_sub}
Let $\clh_k \otimes \cle$ be a $\cle$-valued regular rkHs on $\mathbb{B}_n$ and let $\clm$ be a $M_{\bm{z}}$-joint invariant closed subspace of $\clh_k \otimes \cle$. Then $\clm \neq \{0\}$ if and only if $\clw(M_{\bm{z}}|_{\clm}) \neq \{0\}$.
\end{propn}
\begin{proof}
It is easy to observe that $\bigcap_{l=0}^{\infty} \sum_{|\bm{m}| = l} M_{\bm{z}}^{\bm{m}} \clh_k \otimes \cle = \{0\}$. Otherwise, the order at zero for any function $f \in \bigcap_{l=0}^{\infty} \sum_{|\bm{m}| = l} M_{\bm{z}}^{\bm{m}} \clh_k \otimes \cle$, can be increased arbitrarily and this is possible only when $f = 0$. Thus, $M_{\bm{z}}$ satisfies the assumption of  Proposition \ref{wander_opt} and therefore, we get that $\clw(M_{z}|_{\clm}) \neq \{0\}$ if and only if $\clm \neq \{0\}$. This completes the proof.

\end{proof}

In the setting of the unit ball, we will now prove a  result that is analogous to Lemma \ref{vector2}.

\begin{lemma}\label{cnp_lem}
Let $\clh_k \otimes \cle$ be a $\cle$-valued regular rkHs on $\mathbb{B}_n$ such that 
\begin{enumerate}
\item[(i)] $k/s \geq 0$, where $s$ is a cnp kernel normalized  at $0$;
\item[(ii)] $\clh_s$ is a regular rkHs on $\mathbb{B}_n$.
\end{enumerate}
Let $\clm$ be a proper $M_{\bm{z}}$-invariant closed subspace of $\clh_k \otimes \cle$ such that $\clw(M_{\bm{z}}) \subseteq \clm^{\perp}$. Then there exists a non-zero vector $\eta \in \clm$ such that $M_{z_i}^* \eta \in \clm^{\perp}$ for all $i \in \{1,\ldots,n\}$.
\end{lemma}
\begin{proof}

At first, let us observe that $\clm$ cannot be a $M_{\bm{z}}$-joint reducing subspace of $\clh_k \otimes \cle$ as otherwise, the assumption $\clw(M_{\bm{z}})  = \cle \subseteq \clm^{\perp}$ will contradict that $\clm$ is a proper subspace. 

From Theorem \ref{Beurling_cnp}, it follows that there exists a Hilbert space $\clf$ and a partially isometric multiplier $\Gamma(\bm{z}) \in \mbox{Mult}(\clh_s \otimes \clf, \clh_k \otimes \cle)$ such that
\begin{equation}\label{beurling_cnp}
\clm^{\perp} = \clh_k \otimes \cle \ominus \Gamma (\bm{z}) \big(\clh_s \otimes \clf \big).
\end{equation}
The proof can be divided into the following cases.

Case (i): If there exists a non-zero vector $\eta \in \ran M_{\Gamma}^* \cap \clf$, then for any $f \in \clh_s \otimes \clf$ and $i \in \{1,\ldots,n\}$, we have
\begin{align*}
\langle M_{z_i}^* M_{\Gamma} \eta, M_{\Gamma} f \rangle = \langle M_{\Gamma} \eta, M_{z_i} M_{\Gamma} f \rangle = \langle M_{\Gamma}^* M_{\Gamma} \eta, M_{z_i} f \rangle = \langle \eta, M_{z_i} f \rangle =0.
\end{align*}
This implies that $M_{z_i}^* M_{\Gamma}\eta \in \clm^{\perp}$ for all $i \in \{1,\ldots,n\}$. 

Case (ii): We will obtain a contradiction in the case of $\ran M_{\Gamma}^*\cap \clf = \{0\}$. From equation (\ref{beurling_cnp}) we get that $M_{\Gamma}: \clh_s \otimes \clf \raro \clm$ is a partial isometry and moreover, $\Gamma(\bm{z})$ is a multiplier implies that $M_{\Gamma} M_{z_i} = M_{z_i}|_{\clm} M_{\Gamma}$ for all $i \in \{1,\ldots,n\}$. Now applying Lemma \ref{lemma_wander} gives
\[
\clw(M_{\bm{z}}|_{\clm}) = M_{\Gamma} (\ran M_{\Gamma}^*\cap \clf) = \{0\}.
\]
From Proposition \ref{wand_sub} we have $\clm = \{0\}$, which is a contradiction to our assumption that $\clm$ is proper. This completes the proof.

\end{proof}

We are now ready to establish our main result for rkHs's $H_m(\mathbb{B}_n,\cle)$, for integers $m \geq 1$. First, let us note the following operator identity that has been established in [Lemma 2, \cite{ES}].
\begin{equation}\label{identity}
P_{\cle} = I_{H_m(\mathbb{B}_n, \cle)} - \sum_{j=0}^{m-1} (-1)^j {m \choose j+1} \sum_{| \alpha | = j+1} \gamma_{\alpha} M_{\bm{z}}^{\alpha} M_{\bm{z}}^{*\alpha} \quad (\alpha \in \Nat^n),
\end{equation}
where $\gamma_{\alpha}$ is the coefficient appearing in the following sum
\[
K_m^{-1}(\bm{z},\bar{\bm{w}}) = (1 - \langle \bm{z}, \bar{\bm{w}} \rangle)^{m} I_{\cle} = \sum_{j=0}^m (-1)^j {m \choose j} \sum_{|\alpha| = j}\gamma_{\alpha} \bm{z}^{\alpha} \bm{w}^{\alpha} I_{\cle} \quad (\bm{z},\bm{w} \in \mathbb{B}_n).
\]
Note that $m \geq 1$ implies that the kernel $K_m$ of the rkHs $H_m(\mathbb{B}_n,\cle)$ always admits a cnp factor $s = (1 - \langle \bm{z}, \bm{w} \rangle)^{-1}$, that is the kernel for the Drury-Arveson space.
\begin{thm}\label{Berg_char}
Let $\Phi \in \mbox{Mult}_1(H_m(\mathbb{B}_n,\cle))$, where $m \geq 1$ (and $m \in \Nat$). Then the following are equivalent
\begin{enumerate}
\item[(i)]  $M_{\Phi}$ is a pure contraction on $H_m(\mathbb{B}_n,\cle)$,
\item[(ii)] $\Phi(0)$ is a pure contraction on $\cle$.
\end{enumerate}
\end{thm}
\begin{proof}
$(i)$ implies $(ii)$ is easy to observe since for any $\eta \in \cle$ we have
\[
\lim_{l \raro \infty} \|\Phi(0)^{*l} \eta\| = \lim_{l \raro \infty} \|M_{\Phi}^{*l} K_m(\cdot,0) \eta\| = 0.
\]

$(ii) \implies (i)$: It is easy to observe that $\ker A_{M_{\Phi}}$ is a $(M_{z_1}^*,\ldots,M_{z_n}^*)$-joint invariant closed subspace of $H_m(\mathbb{B}_n, \cle)$. As in the proof of Theorem \ref{main1}, let us focus on the case where $\ker A_{M_{\Phi}}$ is not a $M_{\bm{z}}$-joint reducing subspace of $H_m(\mathbb{B}_n,\cle)$. Furthermore, if we assume that $\ker A_{M_{\Phi}}$ is a proper subspace, then from Lemma \ref{cnp_lem}, there exists a non-zero element $\eta \in (\ker A_{M_{\Phi}})^{\perp}$ such that $M_{z_i}^* \eta \in \ker A_{M_{\Phi}}$ for all $i \in \{1,\ldots,n\}$. Now, for any $l \in \Nat$, there exists $\zeta_l \in \cle$ and $h_{j,l} \in z_j \clh$ such that
\[
M_{\Phi}^{*l} \eta = \zeta_l \oplus \sum_{t=1}^n h_{t,l},
\]
and therefore, $\lim_{l \raro \infty} \|M_{z_i}^* \sum_{t=1}^n h_{t,l} \| = \lim_{l \raro \infty} \|M_{z_i}^* M_{\Phi}^{*l} \eta \| = \lim_{l \raro \infty} \|M_{\Phi}^{*l} M_{z_i}^*  \eta \| = 0$, for any $i \in \{1,\ldots,n\}$. Thus, 
\[
\lim_{l \raro \infty} \|\sum_{j=0}^{m-1} (-1)^j {m \choose j+1} \sum_{|\alpha| = j+1} \gamma_{\alpha} M_{\bm{z}}^{\alpha} M_{\bm{z}}^{*\alpha} (\sum_{t=1}^n h_{t,l} )\| =0. 
\]
From equation (\ref{identity}), we know that
\[
\sum_{j=0}^{m-1} (-1)^j {m \choose j+1} \sum_{|\alpha| = j+1} \gamma_{\alpha} M_{\bm{z}}^{\alpha} M_{\bm{z}}^{*\alpha} = I_{H_m(\mathbb{B}_n, \cle)} - P_{\cle},
\]
and therefore,  
\begin{align*}
\lim_{l \raro \infty} \| \sum_{t=1}^n h_{t,l} \| &= \lim_{l \raro \infty} \|(I_{H_m(\mathbb{B}, \cle)} - P_{\cle}) (\sum_{t=1}^n h_{t,l} ) \| \\
&= \lim_{l \raro \infty} \|\sum_{j=0}^{m-1} (-1)^j {m \choose j+1} \sum_{|\alpha| = j+1} \gamma_{\alpha} M_{\bm{z}}^{\alpha} M_{\bm{z}}^{*\alpha} (\sum_{t=1}^n h_{t,l} )\| = 0.
\end{align*}
Now,
\begin{align*}
\|A_{M_{\Phi}} \eta \| = \lim_{l \raro \infty}  \|A_{M_{\Phi}} M_{\Phi}^{*l} M_{\Gamma} (1 \otimes \eta)\| 
=  \lim_{l \raro \infty}  \|A_{M_{\Phi}}(\zeta_l \oplus \sum_{t=1}^n h_{t,l})\|
&= \lim_{l \raro \infty}  \|A_{M_{\Phi}}(\sum_{t=1}^n h_{t,l})\| \\
& \leq \lim_{l \raro \infty} \|\sum_{t=1}^n h_{t,l}\| = 0,
\end{align*}
implies that $\eta \in \ker A_{M_{\Phi}}$. But $\eta$ was assumed to be an element in $(\ker A_{M_{\Phi}})^{\perp}$ which is possible only if $\eta = 0$. This is a contradiction to our assumption that $\ker A_{M_{\Phi}}$ is a proper subspace. Since $\{0\} \neq \cle \subset \ker A_{M_{\Phi}}$, the only possibility is that $\ker A_{M_{\Phi}} = H_m(\mathbb{B}_n, \cle)$. In other words, $M_{\Phi}$ is a pure contraction on $H_m(\mathbb{B}_n, \cle)$. This completes the proof.
\end{proof}

We can establish the above result for unitarily-invariant cnp spaces as well. Similar to the above situation, we require the following identity analogous to equation (\ref{identity}). This result was established by Chen in [Proposition 2.1, Lemma 2.2, \cite{Chen}].

\begin{propn}[\textit{Chen}]\label{Chen}
Let $\clh_k \otimes \cle$ be an unitarily invariant complete Nevanlinna-Pick space on $\mathbb{B}_n$. Then there exist real numbers $c_{|\beta|}$ such that $c_{|\beta|} \leq  0$ for $|\beta| \geq 1$ and for which
\[
k^{-1}(M_{\bm{z}}, M_{\bm{z}}^*) = \sum_{|\beta|=0}^{\infty} c_{|\beta |} \frac{|\beta| !}{\beta !} M_{\bm{z}}^{\beta} M_{\bm{z}}^{*\beta} = P_{\cle} \quad (\beta \in \Nat^n).
\]
\end{propn}

Note that for any $h \in \clh_k \otimes \cle$ and $n \geq 2$ we have
\begin{align*}
\langle \sum_{1 \leq |\beta| \leq n} c_{|\beta |} \frac{|\beta| !}{\beta !} M_{\bm{z}}^{\beta} M_{\bm{z}}^{*\beta} h, h \rangle - \langle \sum_{1 \leq |\beta| \leq n+1} c_{|\beta |} \frac{|\beta| !}{\beta !} M_{\bm{z}}^{\beta} M_{\bm{z}}^{*\beta} h, h \rangle &= - \langle \sum_{|\beta| = n+1} c_{|\beta |} \frac{|\beta| !}{\beta !} M_{\bm{z}}^{\beta} M_{\bm{z}}^{*\beta} h, h \rangle \\
&= \sum_{|\beta| = n+1} - c_{|\beta |} \frac{|\beta| !}{\beta !} \| M_{\bm{z}}^{*\beta} h \|^2 \\
& \geq 0.
\end{align*}
Thus, $\{ \langle \underset{1 \leq |\beta| \leq n}\sum c_{|\beta |} \frac{|\beta| !}{\beta !} M_{\bm{z}}^{\beta} M_{\bm{z}}^{*\beta} h, h \rangle \}_n$ is a monotonic decreasing sequence of real numbers.

\begin{thm}\label{cnp_char}
Let $\clh_k \otimes \cle$ be an unitarily invariant complete Nevanlinna-Pick space on $\mathbb{B}_n$. Then for any $\Phi \in \mbox{Mult}_1(\clh_k \otimes \cle)$, the following are equivalent
\begin{enumerate}
\item[(i)]  $M_{\Phi}$ is a pure contraction on $\clh_k \otimes \cle$,
\item[(ii)] $\Phi(0)$ is a pure contraction on $\cle$.
\end{enumerate}
\end{thm}
\begin{proof}
As before, $(i)$ implies $(ii)$ is easy to observe. We will prove $(ii)$ implies $(i)$.

Note that $\ker A_{M_{\Phi}}$ is a $(M_{z_1}^*,\ldots,M_{z_n}^*)$-joint invariant closed subspace of $\clh_k \otimes \cle$. Similar to the proof of Theorem \ref{Berg_char}, let us consider the case where $\ker A_{M_{\Phi}}$ is a proper subspace that is not $M_{\bm{z}}$-joint reducing. Then by Lemma \ref{cnp_lem} there exists $\eta \in (\ker A_{M_{\Phi}})^{\perp}$ such that $M_{z_i}^* \eta \in \ker A_{M_{\Phi}}$ for all $i \in \{1,\ldots,n\}$. Now for any $l \in \Nat$, there exists $\zeta_l \in \cle$ and $h_{j,l} \in z_j \clh$ such that
\[
M_{\Phi}^{*l} \eta = \zeta_l \oplus \sum_{j=1}^n h_{j,l},
\]
Thus, $\lim_{l \raro \infty} \|M_{z_i}^* \sum_{j=1}^n h_{j,l} \| = 0$ for any $i \in \{1,\ldots,n\}$ which further implies that
\[
\lim_{l \raro \infty} \|\sum_{|\beta| = 1} c_{|\beta |} \frac{|\beta| !}{\beta !} M_{\bm{z}}^{\beta} M_{\bm{z}}^{*\beta}(\sum_{j=1}^n h_{j,l} )\| =0. 
\]
From Proposition \ref{Chen} it follows that
\[
\sum_{|\beta|=1}^{\infty} c_{|\beta |} \frac{|\beta| !}{\beta !} M_{\bm{z}}^{\beta} M_{\bm{z}}^{*\beta} = I_{\clh_k \otimes \cle} - P_{\cle},
\]
and thus,
\begin{align*}
\lim_{l \raro \infty} \| \sum_{j=1}^n h_{j,l} \| = \lim_{l \raro \infty} \langle \sum_{j=1}^n h_{j,l},  \sum_{j=1}^n h_{j,l}  \rangle &= \lim_{l \raro \infty} \langle (I_{\clh_k \otimes \cle} - P_{\cle}) (\sum_{j=1}^n h_{j,l} ), \sum_{j=1}^n h_{j,l} \rangle \\
&=\lim_{l \raro \infty} \langle \sum_{|\beta|=1}^{\infty} c_{|\beta |} \frac{|\beta| !}{\beta !} M_{\bm{z}}^{\beta} M_{\bm{z}}^{*\beta}(\sum_{j=1}^n h_{j,l} ), \sum_{j=1}^n h_{j,l} \rangle \\
&= \lim_{l \raro \infty} \lim_{m \raro \infty} \langle \sum_{1 \leq |\beta| \leq m}c_{|\beta |} \frac{|\beta| !}{\beta !} M_{\bm{z}}^{\beta} M_{\bm{z}}^{*\beta}(\sum_{j=1}^n h_{j,l} ), \sum_{j=1}^n h_{j,l}\rangle \\
& \leq \lim_{l \raro \infty}  \langle \sum_{|\beta| = 1} c_{|\beta |} \frac{|\beta| !}{\beta !} M_{\bm{z}}^{\beta} M_{\bm{z}}^{*\beta}(\sum_{j=1}^n h_{j,l} ), \sum_{j=1}^n h_{j,l}\rangle\\
&= 0 \quad (\text{since } \|\sum_{j=1}^n h_{j,l}\| \leq \eta \text{, for all } l \in \Nat).
\end{align*}
Now,
\begin{align*}
\|A_{M_{\Phi}} \eta \| = \lim_{l \raro \infty}  \|A_{M_{\Phi}} M_{\Phi}^{*l} \eta\| 
=  \lim_{l \raro \infty}  \|A_{M_{\Phi}}(\zeta_l \oplus \sum_{j=1}^n h_{j,l})\|
&= \lim_{l \raro \infty}  \|A_{M_{\Phi}}(\sum_{j=1}^n h_{j,l})\| \\
& \leq \lim_{l \raro \infty} \|\sum_{j=1}^n h_{j,l}\| = 0,
\end{align*}
This implies that $\eta \in \ker A_{M_{\Phi}}$ and therefore, the only possibility is $\eta = 0$. This is a contradiction to our assumption that $\ker A_{M_{\Phi}}$ is proper which further implies that $\ker A_{M_{\Phi}} = \clh_k \otimes \cle$, in other words, $M_{\Phi}$ is a pure contraction on $\clh_k \otimes \cle$.
\end{proof}

Using Theorem \ref{Berg_char}, we can establish a result that is analogous to Theorem \ref{main1}, with respect to certain special tuples of commuting operators. These tuples are indeed special because in a recent paper \cite{EL}, Eschmeier and Langend\"orfer  established a Wold-von Neumann type decomposition theorem for these tuples. Let us give a  brief overview of these tuples before moving into our results (see \cite{EL} for more details).

Let $X$ be the row operator (denoted by $X \in \clb(\clh)^n$) corresponding to a $n$-tuple of commuting operators $(X_1,\ldots,X_n)$ on $\clh$. Let $\sigma_X: \clb(\clh) \raro \clb(\clh)$ be a positive linear map defined by
\[
\sigma_X(Y): = \sum_{1 \leq i \leq n} X_i Y X_i^* \quad (Y \in \clb(\clh)).
\]
$X$ is said to be \textit{pure} if $\sigma_X^n(I) \raro 0$ in the strong operator topology as $n \raro \infty$.
\begin{defn}
A row operator $X \in \clb(\clh)^n$ is said to be \textit{regular} at $z=0$ if there is a real number $\epsilon > 0$ such that $\|z\| < \epsilon$, the subspace $(X - Z) \clh^n \subset \clh$ is closed and $\clh$ decomposes into the algebraic direct sum 
\[
\clh = (X - Z) \clh^n \oplus \clw(X),
\]
where $Z \in \clb(\clh)^n$ is the row operator defined by $Z (h_i)_{i=1}^n = \sum_{i=1}^n z_i h_i$.
\end{defn}
The following result is a particular case of [Theorem 3.7, Corollary 3.6, \cite{EL}].
\begin{thm}
Let $X \in \clb(\clh)^n$ be a pure commuting $n$-tuple of operators which is regular at $z=0$. Then the following conditions are equivalent:
\begin{enumerate}
\item[(i)] For some positive integer $m \geq 1$, $X$ satisfies the identity $(X^*X)^{-1} = (\oplus \Delta_X)|_{X^* \clh}$, where $\Delta_X = \oplus \sum_{j=0}^{m-1} (-1)^j {m \choose j+1} \sigma_X^j(I_H)$,
\item[(ii)] $X$ is unitarily equivalent to $M_{\bm{z}} \in \clb(H_m(\mathbb{B}_n, \clw(X)))^n$.
\end{enumerate}
\end{thm}
In particular, condition $(i)$ implies that $X$ satisfies the wandering subspace property on $\clh$ and there exists a unitary $U:\clh \raro H_m(\mathbb{B}_n, \clw(X))$ defined by
\[
Uh(\bm{z}) = \sum_{\alpha \in \Nat^n} \gamma_{\alpha} P_{\clw(X)} L^{\alpha}h \hspace{1mm} {\bm{z}}^{\alpha},
\]
where $L:= (X^*X)^{-1} X^*$. Moreover, the unitary $U$ follows the following intertwining property
\begin{equation}\label{intertwine}
UX_i = (M_{z_i} \otimes I_{\cle}) U \quad (i \in \{1,\ldots,n\}),
\end{equation}
and $U(\clw(X)) = \clw(X)$ (see [Theorem 3.5, \cite{EL}] for details). We are now ready to state our main result based on these tuples.
\begin{thm}
Let $X \in \clb(\clh)^n$ be a pure commuting $n$-tuple of operators which is regular at $z=0$ and for some positive integer $m \geq 1$ satisfies $(X^*X)^{-1} = (\oplus \Delta_X)|_{X^* \clh}$. Let $T$ be a contraction on $\clh$ such that $T X_i = X_i T$, for all $i \in \{1,\ldots,n\}$. Then the following are equivalent:
\item[(i)] $T$ is a pure contraction on $\clh$,
\item[(ii)]$P_{\clw(X)} T|_{\clw(X)}$ is a pure contraction on $\clw(X)$.
\end{thm}
\begin{proof}
Note that $(i)$ implies $(ii)$ follows in a manner that is similar to the above results. 

$(ii) \implies (i):$ From equation (\ref{intertwine}), it follows that
\[
UTU^* (M_{z_i} \otimes I_{\cle}) = (M_{z_i} \otimes I_{\cle}) UTU^* \quad (i \in \{1,\ldots,n\}).
\]
This implies that there exists a contractive multiplier $\Theta \in \mbox{Mult}_1(H_m(\mathbb{B}_n, \clw(X)))$ such that $T = U^* M_{\Theta} U$ (see Proposition 4.2, \cite{jaydeb3}). Now assumption $(b)$ implies that $\clw(X) \subseteq \ker A_{T}$, and therefore, for any $\eta \in \clw(X)$ we have
\[
\lim_{n \raro \infty} \|M_{\Theta}^{*n} U \eta \| = \lim_{n \raro \infty} \|U T^{*n}  \eta\| = \lim_{n \raro \infty} \|T^{*n}  \eta \| =0.
\]
Thus, $\clw(X) = U \clw(X) \subseteq \ker A_{M_{\Theta}}$ and therefore, Theorem \ref{Berg_char} implies that $M_{\Theta}$ is a pure contraction on $H_m(\mathbb{B}_n, \clw(X))$. And by unitary equivalence, we get that $T$ is also a pure contraction on $\clh$. This completes the proof.
\end{proof}

\section{Applications}\label{Sec5}
Recently, the author proved that pair of commuting pure contractions with finite dimensional defect spaces admit pure isometric dilation (see \cite{SS}). The aim of this section is to apply Theorem \ref{main3} to extend the above result for certain important tuples of commuting pure contractions appearing in recent papers \cite{BDHS} and \cite{BDS}.

Let 
\[
\clx^n(\clh) := \{(X_1,\ldots,X_n) \in \clb(\clh): \|X_ih\| \leq 1, X_iX_j = X_jX_i, 1 \leq i,j \leq n\}
\]
A tuple $X \in \clx^n(\clh)$ is said to be \textit{pure} if $X_i$ is a pure contraction for all $i \in \{1,\ldots,n\}$. 

With respect to the Szeg\"o kernel $\mathbb{S}_n(\bm{z}, \bm{w})$ of the Hardy space $H^2(\D^n)$, let
\[
\mathbb{S}_n(\clh): = \{X \in \clx^n(\clh): \mathbb{S}_n^{-1}(X,X^*) \geq 0\}.
\]

For the sequel, we need a special type of operator-valued multipliers called the \textit{Schur-Agler} class on $\D^n$ (see \cite{AM}). 

\begin{defn}
For a Hilbert space $\cle$, the Schur-Agler class on $\D^n$ (denoted by $\cls \cla_n(\cle,\cle)$) is defined to be the collection of all $\clb(\cle)$-valued bounded analytic function on $\D^n$ for which there exists Hilbert spaces $\{\clh_i\}_{i=1}^n$ and a unitary 
\begin{equation}
U= \begin{bmatrix}\label{transfer_uni}
A & B \\
C & D
\end{bmatrix}: \cle \oplus (\underset{i=1}{\overset{n}\oplus} \clh_i) \raro \cle \oplus (\underset{i=1}{\overset{n}\oplus} \clh_i),
\end{equation}
such that
\[
\Phi(\bm{z}) := A + B E(\bm{z})(I - DE(\bm{z}))^{-1}C \quad (\bm{z} \in \D^n),
\]
where $E(\bm{z}) := \oplus_{i=1}^n z_i I_{\clh_i}$.
\end{defn}
It follows from a well-known computation that any element in $\cls \cla_n(\cle,\cle)$ is contractive.

Let us begin by observing a result that follows from [Theorem 4.2 and Proposition 4.1, \cite{BDHS}]. For the sake of completeness we have included the proof of the result.

\begin{lemma}[\cite{BDHS},\cite{BDS}]\label{cnu_symbol}
Let $T$ be a contraction on a Hilbert space $\clh$ for which there exists a Hilbert space $\cle$, $\Psi \in \cls \cla_n(\cle,\cle)$ and an isometry $\Pi:\clh \raro H_{\cle}^{2}(\D^n)$ (where, $n \geq 1$), such that $\Pi T^* = M_{\Psi}^* \Pi$ and $H_{\cle}^{2}(\D^n) = \bigvee_{\bm{k} \in \Nat^n} M_{\bm{z}}^{\bm{k}} \Pi \clh$. If $T$ is a pure contraction on $\clh$, then $\Psi(0)$ is a c.n.u. contraction on $\clk$.
\end{lemma}
\begin{proof}
By definition, $\Psi \in \cls \cla_n(\cle,\cle)$ implies that there exist Hilbert spaces $\{ \clh_i \}_{i=1}^n$ and a unitary $U$ of the form in equation (\ref{transfer_uni}) such that
\[
\Phi(\bm{z}) = A + B E(\bm{z})(I - DE(\bm{z}))^{-1}C \quad (\bm{z} \in \D^n).
\]
By the Nagy-Foias decomposition of a contraction there exists an orthogonal decompostion $\cle = \cle_{0} \oplus \cle_{1}$ such that 
\[
A = \begin{bmatrix}
A|_{\cle_0} & 0 \\
0 & A|_{\cle_1} 
\end{bmatrix}: \cle_0 \oplus \cle_1 \raro \cle_0 \oplus \cle_1,
\]
where, $A|_{\cle_{0}}$ is a unitary and $A|_{\cle_1}$ is a c.n.u. contraction. With respect to this decomposition of $A$, the transfer function $\Psi(\bm{z})$ can be decomposed in the following manner (see [Proposition 3.1, \cite{BDS}])
\[
\Psi(z) = \begin{bmatrix}
\Psi_{0}(\bm{z}) & 0 \\
0 & \Psi_{1}(\bm{z})
\end{bmatrix} \quad (\bm{z} \in \D^n),
\]
where, $\Psi_0(\bm{z}) = A|_{\cle_0}$ and $\Psi_1(\bm{z}) = A|_{\cle_1} + B E(\bm{z})(I - DE(\bm{z}))^{-1}C \in \cls_n(\cle_1,\cle_1)$. The proof will be complete if we can show that $\Pi \clh \subseteq H_{\cle_{1}}^{2}(\D^n)$, since then $H_{\cle}^2(\D^{n}) = \bigvee_{\bm{k} \in \Nat^n} M_{\bm{z}}^{\bm{k}} \Pi \clh \subseteq H_{\cle_{1}}^{2}(\D^n) \subseteq H_{\cle}^2(\D^{n})$ will imply that $\cle_0 = \{0\}$. Now $g \in H_{\cle_0}^2(\D^n)$ implies that there exist $h_m \in H_{\cle_0}^2(\D^n)$ for all $m \in \Nat$ such that $g = M_{\Psi}^m h_m$ (since $M_{\Psi_0}$ is a unitary on $H_{\cle_0}^2(\D^n)$). Thus, for $f \in \Pi \clh$ we have 
\[
| \langle f,g \rangle | = | \langle f, M_{\Psi}^{m} h_m \rangle | = | \langle M_{\Psi}^{*m} \Pi h, h_m \rangle | = | \langle \Pi X^{*m}h, h_m \rangle | \leq \|X^{*n}h\| \|g\|,
\]
Taking limit on both sides as $m \raro \infty$ gives $\langle f,g \rangle = 0$ which implies that $\Pi \clh \perp H_{\cle_0}^2(\D^{n})$. This completes the proof.
\end{proof}

\subsection{Tuples associated to von Neumann inequality in polydisc}
Let us begin by giving a brief introduction on the tuples of operators appearing in \cite{BDHS}. A particular class of these tuples were first studied by Grinshpan et al. (\cite{GKVW}) in connection to multivariate von Neumann inequality. For each $i, p, q \in \{1,\ldots,n\}$ such that $p \neq q$, let
\[
\hat{X}_i := \{X \in \clx^n(\clh): (X_1,\ldots,X_{i-1},X_{i+1},\ldots,X_n) \in \clx^{n-1}(\clh)\}.
\]
and
\[
\clx_{p,q}^n(\clh) := \{X \in \clx^n(\clh): \hat{X}_p, \hat{X}_q \in \mathbb{S}_{n-1}(\clh) \text{ and } \hat{X}_p \text{ is pure}\}.
\]
$X \in \clx_{p,q}(\clh)$ is said to be of \textit{finite rank} if 
\[
\cld_{\hat{X}_{i}} := \overline{\mbox{ran}} \hspace{1mm}\mathbb{S}^{-1}(\hat{X}_i,{\hat{X}_{i}}^*)^{\frac{1}{2}},
\]
is a finite dimensional subspace of $\clh$, for $i=p,q$.
The following result on isometric dilation for these tuples was obtained in [Theorem 5.3, \cite{BDHS}].
\begin{thm}\label{BDS1}
Let $X=(X_1,\ldots,X_n) \in \clx_{p,q}^n(\clh)$ be finite rank, then there exists a Hilbert space $\cle$, inner multiplies $\Phi_p,\Phi_q \in H_{\cle}^{\infty}(\D^{n-1})$
and an isometry $\Pi: \clh \raro H_{\cle}^2(\D^{n-1})$ such that
\[
\Pi X_i^* = \left\{
                \begin{array}{ll}
                  M_{z_i}^* \Pi \quad (i \in \{1,\ldots,n-1\} \setminus \{p,q\}),\\
                  M_{\Phi_i}^* \Pi \quad (i=p,q),
               \end{array}
              \right. 
\]
where $\Phi_p,\Phi_q$ are inner polynomials of degree atmost one such that
\[
\Phi_p(\bm{z}) \Phi_q(\bm{z}) = \Phi_q(\bm{z}) \Phi_p(\bm{z}) = z_p I_{\cle} \quad (\bm{z} \in \D^{n-1}).
\]
In other words, $X$ dilates to a $n$-tuple of commuting isometries 
\[
(M_{z_1},\ldots,M_{z_{p-1}}, M_{\Phi_p},M_{z_{p+1},}  \ldots, M_{z_{q-1}}, M_{\Phi_q}, M_{z_{q+1}}, \ldots,M_{z_n})
\]
on $H_{\cle}^2(\D^{n-1})$.
\end{thm}
A remarkable contribution of the above result lies in the fact that $M_{\Phi_p},M_{\Psi_q}$ are explicitly determined to be of Berger, Coburn and Lebow type isometries \cite{BCL}. In particular, there exists a collection $(\cle,U,P)$ where, $\cle$ is a Hilbert space, $U$ is a unitary and $P$ is an orthogonal projection such that
\begin{equation}\label{BCL}
\Phi_p(\bm{z}) = (P+z_pP^{\perp})U^*; \quad \Phi_q(\bm{z}) = U(P^{\perp} + z_p P).
\end{equation}
Using Theorem \ref{main3}, we can establish the following result on pure isometric dilation.
\begin{thm}
Let $X=(X_1,\ldots,X_n)$ be a tuple of commuting pure contractions which belongs to the collection $\clx_{p,q}^n(\clh)$ and has finite rank. Then there exists a Hilbert space $\cle$ such that $X$ dilates to a $n$-tuple of commuting pure isometries 
\[
(M_{z_1},\ldots,M_{z_{p-1}}, M_{\Phi_p},M_{z_{p+1},}  \ldots, M_{z_{q-1}}, M_{\Phi_q}, M_{z_{q+1}}, \ldots,M_{z_n})
\]
on $H_{\cle}^2(\D^{n-1})$.
\end{thm}
\begin{proof}
Using Lemma \ref{cnu_symbol} and Theorem \ref{BDS1} we obtain that $\Phi_p(0), \Psi_{q}(0)$ are c.n.u. contractions on the finite dimensional space $\cle$, which further implies that $\Phi(0),\Psi(0)$ are pure contractions on $\cle$ (see \cite{NF}). Now, Theorem \ref{main3} implies that $M_{\Phi_p}, M_{\Psi_q}$ are pure isometries. This completes the proof.
\end{proof}

\subsection{Tuples associated to commutant lifting in unit polydisc}
In the same manner as above, let us begin with a brief overview of tuples of contractions appearing in \cite{BDS}. These tuples were first studied by Ball et al. in their influential paper \cite{BLTT} on the commutant lifting problem in the polydisc. Given $A \in \clb(\clh)$, a conjugate map $C_A:\clb(\clh) \raro \clb(\clh)$ is the completely positive map defined by
\[
C_A(Y) := AYA^* \quad (Y \in \clb(\clh)).
\]

\begin{defn}
$\clp^n(\clh)$ is defined to be the collection of $n$-tuples of operators $X \in \clx^n(\clh)$ that satisfy the following conditions:
\begin{enumerate}
\item[(i)] $\hat{X}_n \in \mathbb{S}_{n-1}(\clh)$;
\item[(ii)] There exist positive operators $G_1,\ldots,G_{n-1} \in \clb(\clh)$ (depending on $X$) for which
\begin{enumerate}
\item[(a)] $I - X_n X_n^* = G_1 + \ldots + G_{n-1}$;
\item[(b)] $ S_X(G_i):= \prod_{j=1; j \neq i}^{n-1} (I_{\clb(\clh)} - C_{X_j}) G_i \geq 0$ for all $i \in \{1,\ldots,n\}$. 
\end{enumerate}
are satisfied.
\end{enumerate}
\end{defn}
The following result on isometric dilation of these tuples was obtained by Barik et al. in [Theorem 4.4, \cite{BDS}].

\begin{thm}\label{BDS2}
If $X \in \clp_n(\clh)$ is of finite rank, then there exists an inner multiplier $\Phi \in \cls \cla_{n-1}(\cld_{\hat{X}_n},\cld_{\hat{X}_n})$ and an isometry $\Pi: \clh \raro H_{\cld_{\hat{X}_n}}^2(\D^{n-1})$ such that
\[
\Pi X_i^* = \left\{
                \begin{array}{ll}
                  M_{z_i}^* \Pi \quad (i=1,\ldots,n-1),\\
                  M_{\Phi}^* \Pi \quad (i=n).
                \end{array}
              \right. 
\]
In other words, every finite rank $X \in \clp_n(\clh)$ dilates to $(M_{z_1},\ldots,M_{z_{n-1}}, M_{\Phi})$ on $H_{\cld_{\hat{X}_n}}^2(\D^{n-1})$.
\end{thm}
The precise structure of the multiplier $\Phi \in \cls \cla_{n-1}(\cld_{\hat{X}_n}, \cld_{\hat{X}_n})$ is the following:\\ let $F_i^2 := S_{X}(G_i)$ for $i=1,\ldots,n-1$ and $\clf_i = \overline{ran} F_i$. The unitary in our discussion is $U:\cld_{\hat{X}_n} \oplus (\oplus_{i=1}^{n-1} \clf_i) \raro \cld_{\hat{X}_n} \oplus (\oplus_{i=1}^{n-1} \clf_i)$ such that
\begin{equation}\label{unitary_3}
U(D_{\hat{X}_n}h, F_1 X_1^*h,\ldots,F_{n-1}X_{n-1}^* h) = (D_{\hat{X}_n}X_n^*h, F_1 h, \ldots, F_n h).
\end{equation}
Let
\[
U= \begin{bmatrix}
A & B \\
C & D
\end{bmatrix}: \cld_{\hat{X}_n} \oplus (\oplus_{i=1}^{n-1} \clf_i) \raro \cld_{\hat{X}_n} \oplus (\oplus_{i=1}^{n-1} \clf_i),
\]
and $E(\bm{z}) := \oplus_{i=1}^{n-1} z_i I_{\clf_i} \in \clb(\oplus_{i=1}^{n-1} \clf_{i})$. Then $\Phi(\bm{z})$ has the realization in the following manner,
\[
\Phi(\bm{z}) := \tau_{U}(\bm{z}) = A + B E(\bm{z})(I - DE(\bm{z}))^{-1}C \quad (\bm{z} \in \D^n),
\]

Using Theorem \ref{main3}, we obtain the following result on pure isometric dilation.
\begin{thm}
Let $X=(X_1,\ldots,X_n)$ be a tuple of commuting pure contractions which belongs to the collection $\clp_n(\clh)$ and has finite rank. Then there exists a pure inner multiplier $\Psi(z) \in \cls \cla_{n-1}(\cld_{\hat{X}_n},\cld_{\hat{X}_n})$ such that $X$ dilates to a $n$-tuple of commuting pure isometries $(M_{z_1}, \ldots, M_{z_{n-1}}, M_{\Phi})$ on $H_{\cld_{\hat{X}_{n}}}^2(\D^{n-1})$.
\end{thm}
\begin{proof}
Using Lemma \ref{cnu_symbol} and Theorem \ref{BDS2} we obtain that $\Phi(0)$ is a c.n.u. contraction on the finite dimensional Hilbert space $\cld_{\hat{X}_{n}}$ and therefore, $\Phi(0)$ is a pure contraction on $\cld_{\hat{X}_{n}}$. Now, Theorem \ref{main3} shows that $M_{\Phi}$ is a pure isometry on $H_{\cld_{\hat{X}_n}}^2(\D^{n-1})$. This finishes the proof.
\end{proof}

\section*{Acknowledgements}
The author would like to thank Prof. B.K. Das for useful discussions and he is grateful to Prof. Kunyu Guo for a correspondence regarding Proposition \ref{wand_sub}. He would also like to thank Prof. Gautam Bharali for going through the initial drafts of this article and providing important feedback. The author is supported by DST-INSPIRE Faculty Fellowship No. - DST/INSPIRE/04/2019/000769.

\end{document}